\documentclass[a4paper]{amsart}

\RequirePackage{amsmath} 
\RequirePackage{amssymb}
\usepackage{amscd,latexsym,amsthm,amsfonts,amssymb,amsmath,amsxtra}
\usepackage[colorlinks=true,urlcolor=blue,citecolor=blue]{hyperref}
\usepackage{color}
\usepackage[all]{xy}
\usepackage{bm}
\usepackage{mathtools}
\usepackage{ mathrsfs }
\usepackage{xcolor}
\usepackage{comment}
\usepackage{marginnote}
\usepackage{enumitem}
\usepackage{thmtools, thm-restate}

\let\Re\undefined
\let\Im\undefined

\DeclareMathOperator{\Re}{Re}
\DeclareMathOperator{\Im}{Im}

\DeclareMathOperator{\Tr}{Tr}

\newcommand{\floor}[1]{{\left\lfloor#1\right\rfloor}}

	\newcommand{\Res}{\operatorname{Res}}

	\newcommand{\sgn}{\operatorname{sgn}}

	\newcommand{\fin}{\operatorname{fin}}

	\newcommand{\diag}{\operatorname{diag}}

	\newcommand{\Vol}{\operatorname{Vol}}

	\newcommand{\Ind}{\operatorname{Ind}}
	
	\newcommand{\RNum}[1]{\uppercase\expandafter{\romannumeral #1\relax}}

\begin{document}
\theoremstyle{plain}
\newtheorem{thm}{Theorem}[section]
	
\newtheorem{cor}[thm]{Corollary}

\newtheorem{thmx}{Theorem}
\renewcommand{\thethmx}{\Alph{thmx}} 

\newtheorem{hy}[thm]{Hypothesis}
\newtheorem*{thma}{Theorem A}
\newtheorem*{corb}{Corollary B}
\newtheorem*{thmc}{Theorem C}
\newtheorem{lemma}[thm]{Lemma}  
\newtheorem{prop}[thm]{Proposition}
\newtheorem{conj}[thm]{Conjecture}
\newtheorem{fact}[thm]{Fact}
\newtheorem{claim}[thm]{Claim}
	
\theoremstyle{definition}
\newtheorem{defn}[thm]{Definition}
\newtheorem{example}[thm]{Example}
\newtheorem{que}[thm]{Question}
\newtheorem{ass}[thm]{Assumption}

\theoremstyle{remark}
	
\newtheorem{remark}[thm]{Remark}	
\numberwithin{equation}{section}

\title[]{Fourier Spectral Reciprocity and Canonical Hecke $L$-Functions}
 
\author{Liyang Yang}

\address{Department of Mathematics, Texas A\&M University, College Station,  TX 77843, USA} 
\email{liyangy@tamu.edu} 

\begin{abstract}
We prove a Fourier-type toric spectral reciprocity formula over number fields,
relating Tate integrals over a quadratic extension to a dual family of Hecke
periods over the base field. The formula keeps the spectral parameter free and
allows test functions adapted to arithmetic applications.

We apply this reciprocity formula to canonical Hecke $L$-functions over CM
fields and obtain explicit first moment formulas, including both
central values and central derivatives. These formulas are uniform in the weight and the twisting conductor, and require
no Heegner-type splitting hypothesis.

Together with subconvexity bounds, they yield quantitative nonvanishing results
for canonical Hecke $L$-functions and their derivatives. In weight one, this gives a rank-one analogue of the arithmetic applications of
Masri--Yang, yielding quantitative Mordell--Weil rank-one results for quadratic
twists of the associated CM abelian varieties.
\end{abstract}

\date{\today}
\maketitle
\tableofcontents

\section{Introduction}

Canonical Hecke characters, originating in the work of Rohrlich
\cite{Roh80a, Roh80b} and brought into arithmetic prominence by Gross, form a
distinguished class of algebraic Hecke characters.  We shall refer
to the $L$-functions attached to them as \textit{canonical Hecke $L$-functions}.
In the imaginary quadratic case, Gross \cite{Gro80} associated a canonical
weight-one Hecke character to an elliptic $\mathbb{Q}$-curve over the Hilbert
class field and conjectured an exact formula for its Mordell--Weil rank in terms
of the sign of the functional equation.  This conjecture was resolved through
nonvanishing results for canonical Hecke $L$-functions
\cite{Roh80a, Roh80b, Roh82, MR82, MY00}, together with deep arithmetic theorems
\cite{Rub81, KL89}.

For higher weights and CM fields, the corresponding nonvanishing problems remain central to the arithmetic and geometry of canonical Hecke characters.  Under
suitable hypotheses, Masri and Yang \cite{MY11} established first moment
asymptotics for central values of canonical Hecke $L$-functions over CM fields,
with applications to ranks of CM abelian varieties.  Their method relies on
period formulas and the equidistribution of CM points.  While powerful, this
geometric approach is tied to the central point: central values and central
derivatives are accessed through different period or height formulas.  This
suggests the need for an analytic framework in which the central parameter
remains free.

In the imaginary quadratic case, such an analytic approach was carried out by
Templier \cite{Tem10}, who obtained first moment formulas for canonical Hecke
$L$-functions and thereby gave an alternative route to the results of
\cite{KMY11}.  Over general CM fields, however, the special one-dimensional
features used in this setting, such as explicit sums over integers, are no
longer available.

The purpose of this paper is to develop a different approach, uniform over
quadratic extensions of number fields. We prove a new spectral reciprocity
formula for Hecke $L$-functions over such extensions, with test functions
adapted to arithmetic applications. By suitable specialization, this gives
explicit first moment formulas for canonical Hecke $L$-functions over CM
fields, including both central values and central derivatives. These formulas
lead to quantitative nonvanishing results in hybrid ranges, recovering and
strengthening the central-value applications of Masri--Yang \cite{MY11} by
relaxing local hypotheses and allowing the parameters to vary. They also make central derivatives accessible and yield rank-one results for quadratic twists of the associated CM abelian varieties, including higher-dimensional analogues of Gross's elliptic $\mathbb{Q}$-curves
$A(\chi)$ \cite{Gro80}.

Although the arithmetic applications in this paper focus on canonical Hecke
characters, the reciprocity formula itself is more flexible. With suitable
choices of local test functions, the same framework applies to algebraic
anticyclotomic Hecke characters over CM fields, with the canonical family
serving as a distinguished arithmetic specialization.

\subsection{Fourier Spectral Reciprocity}
Let $E/F$ be a quadratic extension of number fields, and let $\eta$ be the associated quadratic character. Fix an embedding $\iota: E^{\times}\hookrightarrow  \mathrm{GL}_2(F)$, regarding $E^{\times}$ as a nonsplit torus.

Let $\mu$ and $\omega$ be characters of
$F^{\times}\backslash\mathbb{A}_F^{\times}$, and let $s\in\mathbb{C}$. Let $h(\cdot,s)$ be a section in the induced representation $\pi=\mu|\cdot|^s\boxplus \overline{\mu}\omega|\cdot|^{-s}$. Define
\begin{align*}
h^{\Diamond}(g,s):=\int_{N(\mathbb{A}_F)}h(wug,s)du,\quad w=\begin{pmatrix}
& 1\\
1
\end{pmatrix}.
\end{align*}

For a Hecke character $\chi$ of $E^{\times}\backslash\mathbb{A}_E^{\times}$ satisfying $\chi|_{\mathbb{A}_F^{\times}}=\overline{\omega}$, define
\begin{align*}
Z(\chi;h(\cdot,s)):=\int_{\mathbb{A}_F^{\times}\backslash\mathbb{A}_E^{\times}}h(\iota(x),s)\chi(x)d^{\times}x.
\end{align*}

Then $Z(\chi;h(\cdot,s))$ is a Tate integral representing the completed Hecke $L$-function $\Lambda(1/2+s,\widetilde{\mu}\chi)$, where $\widetilde{\mu}:=\mu\circ\mathrm{Nr}_{E/F}$. It converges absolutely for $\Re(s)$ sufficiently large and admits a meromorphic continuation to $\mathbb{C}$. We continue to denote this meromorphic continuation by $Z(\chi;h(\cdot,s))$.

\begin{restatable}[]{thmx}{SR}\label{thm3.1}
Let $|\Re(s)|<1/2$. Then
\begin{equation}\label{1.1}
\sum_{\chi\in \widehat{C}_E:\
\chi|_{\mathbb{A}_F^{\times}}=\overline{\omega}}\frac{Z(\chi;h(\cdot,s))}{L(1,\eta)}=J_{\mathrm{Const}}(h(\cdot,s))+J_{\mathrm{Res}}(h(\cdot,s))+J_{\mathrm{Dual}}(h(\cdot,s)),
\end{equation}
where
\begin{align*}
&J_{\mathrm{Const}}(h(\cdot,s)):=h(I_2,s)+h^{\Diamond}(I_2,s),\\
&J_{\mathrm{Res}}(h(\cdot,s)):=\big[h(w,s)+h^{\Diamond}(w,s)\big]\prod_{v\in \Sigma_{F,\infty}}\Gamma_{F_v}(1),\\
&J_{\mathrm{Dual}}^{\heartsuit}(h(\cdot,s)):=\frac{1}{2\pi\underset{\lambda=1}{\Res}\ \zeta_F(\lambda)}\sum_{\xi\in \widehat{C}_F}\int_{\mathbb{R}}\widetilde{\Psi}(W(\cdot,s),\xi|\cdot|^{it})dt.
\end{align*}
\end{restatable}

Here $\Gamma_{F_v}(\cdot)$ is the local Gamma function defined by \eqref{eq1.1},  $\widetilde{\Psi}(W(\cdot,s),\xi|\cdot|^{it})$ denotes the meromorphic Rankin--Selberg convolution representing the completed twisted $L$-function $\Lambda(1/2+it,\pi\otimes\xi)$, and $\widehat{C}_F$ is the Pontryagin dual of $F^{\times}\backslash\mathbb{A}_F^{(1)}$.

We call \eqref{1.1} a \textit{Fourier-type toric spectral reciprocity formula}: it gives a direct identity between Tate integrals over $E$, arising from a nonsplit toric Fourier expansion, and Hecke periods over $F$ of Whittaker functions arising from the Fourier expansion of the Eisenstein series attached to $h(\cdot,s)$. 

\subsection{The First Moment Formula}Let $E/F$ be a CM extension with CM type $\boldsymbol{\Phi}$, and suppose that
every prime of $F$ above $2$ is unramified in $E$.  Let
$k\in \mathbb{Z}_{\geq 0}$, and let
$X_{E/F}^{\mathrm{un}}(k,\boldsymbol{\Phi})$ denote the set of \textit{unitary}
canonical Hecke characters of $E$ of type $(2k+1)\boldsymbol{\Phi}$.

Let $\mu$ be a unitary quadratic Hecke character of
$F^{\times}\backslash\mathbb{A}_F^{\times}$ with arithmetic conductor
$\mathfrak{f}$ coprime to $2\mathfrak{D}_E$, where $\mathfrak{D}_E$ denotes the
different, and put $\widetilde{\mu}:=\mu\circ N_{E/F}$. 

Let $[F:\mathbb{Q}]=d$ and $\epsilon_{\mu,k}=\mu_{\infty}(-1)(-1)^{kd}$.
Define
\begin{equation}\label{a}
\mathcal{M}_{\mathfrak{f}}(s,k):=\frac{2^{sd}D_{E/F}^{1/2+s}D_F^{1/2+s}N_F(\mathfrak{f})^{s}L^{(\mathfrak{f})}(1+2s,\eta)}{\pi^{d}},
\end{equation}
where $D_{E/F}$ is the norm of the relative discriminant, $D_F$ is the absolute
discriminant of $F$, and $L^{(\mathfrak{f})}(1+2s,\eta)$ denotes the partial
$L$-function with the Euler factors at primes dividing $\mathfrak{f}$ removed.

By choosing a suitable test section $h(\cdot,s)$ in the spectral reciprocity
formula \eqref{1.1}, we obtain the following asymptotic formula.

\begin{thmx}\label{thma}
Let $0<\delta<1/2$ and $|\Re(s)|\leq \delta$. Then
\begin{align*}
\sum_{\chi\in X_{E/F}^{\mathrm{un}}(k,\boldsymbol{\Phi})}
L(1/2+s,\widetilde{\mu}\chi)
=
\frac{
\mathcal{M}_{\mathfrak{f}}(s,k)
+\epsilon_{\mu,k}\mathcal{M}_{\mathfrak{f}}(-s,k)}
{2^{sd}D_{E/F}^{s}D_F^{s}N_F(\mathfrak{f})^{s}}\cdot L(1,\eta)
+\mathcal{E}_{\mu}(s,k),
\end{align*}
where 
\begin{equation}\label{e1.3}
\mathcal{E}_{\mu}(s,k)
\ll
D_{E/F}^{\frac{7}{16}+2\delta+\varepsilon}
(1+k+|s|)^{\frac{15d}{16}+2\delta d+\varepsilon}
N_F(\mathfrak{f})^{\frac{11}{16}+\delta+\varepsilon}L(1,\eta).
\end{equation}
The implied constant depends only on $F$, $\delta$, and $\varepsilon$.
\end{thmx}

\begin{remark}[Sources of Main Terms]
In \eqref{1.1}, two of the four degenerate terms vanish. The remaining terms
$h(I_2,s)$ and $h^{\Diamond}(w,s)$ give
$\mathcal{M}_{\mathfrak{f}}(s,k)$ and
$\epsilon_{\mu,k}\mathcal{M}_{\mathfrak{f}}(-s,k)$, respectively.
The second contribution is somewhat surprising: $h^{\Diamond}(w,s)$ comes from
the nonconstant Fourier terms of a suitable Eisenstein series, which in many
related settings, such as the class group character family in \cite{Yan20},
decay exponentially. Its survival here reflects a subtle feature of the
canonical-character family. The sign $\epsilon_{\mu,k}$ reflects the
functional-equation symmetry.
\end{remark}

\begin{remark}[Central Values and Derivatives]
Masri--Yang \cite{MY11} proved a first moment of
$L(1/2,\widetilde{\mu}\chi)$ over CM fields, for a slightly different family of
characters, using the period formula of \cite{RY99}. Since that method is tied
to central values, it does not apply directly to
$L(1/2+s,\widetilde{\mu}\chi)$ and therefore gives no information on
$L'(1/2,\widetilde{\mu}\chi)$. One advantage of Theorem \ref{thma} is that the
spectral parameter remains free, making central derivatives accessible. This
is crucial for our arithmetic application to the rank-one case of Gross's
conjecture over CM fields.
\end{remark}

\begin{remark}[Quantitative Comparison]
Under a Heegner-type splitting hypothesis on $\mathfrak{f}$, the main
asymptotic of Masri--Yang \cite{MY11} follows from equidistribution of a toric
suborbit of CM points on a Hilbert modular variety, but without a quantitative
error term; such a quantitative version can now be obtained from
\cite[Theorem B]{HY26}. Masri--Yang also obtain a quantitative variant, via
Venkatesh's ergodic method \cite{Ven10}, under an additional small split-prime
hypothesis. In contrast, Theorem \ref{thma} gives an unconditional quantitative
twisted formula directly from spectral reciprocity. Since the proof does not
use CM points, no Heegner-type splitting hypothesis is required.
\end{remark}

For $\chi\in X_{E/F}^{\mathrm{un}}(k,\boldsymbol{\Phi})$, the root number of
$L(1/2,\widetilde{\mu}\chi)$ is $\epsilon_{\mu,k}$.  Taking $s=0$ in Theorem
\ref{thma} gives the following central value and central derivative formulas.

\begin{cor}\label{cor1.4}
Suppose that $\epsilon_{\mu,k}=1$. Let $0<\varepsilon<10^{-3}$. Then
\begin{equation}\label{1.3}
\frac{1}{h_{E/F}}\sum_{\chi\in X_{E/F}^{\mathrm{un}}(k,\boldsymbol{\Phi})}
L(1/2,\widetilde{\mu}\chi)
=
2^{d}L^{(\mathfrak{f})}(1,\eta)
+\widetilde{\mathcal{E}}_{\mu}(0,k),
\end{equation}
where
\begin{align*}
\widetilde{\mathcal{E}}_{\mu}(0,k)
\ll_{F,\varepsilon}
D_{E/F}^{-1/16+\varepsilon}
(1+k)^{15d/16+\varepsilon}
N_F(\mathfrak{f})^{11/16+\varepsilon}.
\end{align*}
\end{cor}

\begin{remark}
When $F=\mathbb{Q}$, \eqref{1.3} recovers \cite[Theorem 1]{Tem10} and
\cite[Theorem 1.1]{KMY11}, with explicit dependence on $k$ and
$N_F(\mathfrak{f})$. It also gives a new proof of these formulas in this case.
\end{remark}

\begin{cor}\label{cor1.6}
Suppose $\epsilon_{\mu,k}=-1$. Let $0<\varepsilon<10^{-3}$. Then
\begin{multline}\label{1.4}
\frac{1}{h_{E/F}}\sum_{\chi\in X_{E/F}^{\mathrm{un}}(k,\boldsymbol{\Phi})}L'(1/2,\widetilde{\mu}\chi)
=2^{d}
\Big[L^{(\mathfrak f)}(1,\eta)\log(2^dD_{E/F}D_FN_F(\mathfrak f))\\
+2(L^{(\mathfrak f)})'(1,\eta)\Big]+\widetilde{\mathcal{E}}_{\mu}'(0,k),
\end{multline}
where 
\begin{align*}
\widetilde{\mathcal{E}}_{\mu}'(0,k)\ll_{F,\varepsilon} D_{E/F}^{-1/16+\varepsilon}
(1+k)^{15d/16+\varepsilon}N_F(\mathfrak{f})^{11/16+\varepsilon}.
\end{align*}
\end{cor}

The central derivative formula \eqref{1.4} appears to be new even when
$F=\mathbb{Q}$.

\subsection{Effective Quantitative Nonvanishing}

When $F=\mathbb{Q}$, the work of \cite{LX04} implies that the vanishing order 
of $L(s,\widetilde{\mu}\chi)$ at $s=1/2$ is either $0$ or $1$, provided
$N_F(\mathfrak{f})\ll D_{E}^{1/12}$.  One important application of Corollaries
\ref{cor1.4} and \ref{cor1.6} is the existence, for each $j\in \{0,1\}$, of a
character $\chi\in X_{E/F}^{\mathrm{un}}(k,\boldsymbol{\Phi})$ such that
$L^{(j)}(1/2,\widetilde{\mu}\chi)\neq 0$.  Moreover, in the imaginary quadratic
case, if $(2k+1, h_{E/F})=1$, then the work of \cite{Roh80, Shi76} shows that
the nonvanishing of $L(1/2,\widetilde{\mu}\chi)$ for one
$\chi\in X_{E/F}^{\mathrm{un}}(k,\boldsymbol{\Phi})$ implies the nonvanishing of
all central values in the family.  In general, such a one-implies-all principle
is not available.  Nevertheless, quantitative nonvanishing can still be obtained
from subconvexity bounds, or from bounds for $\ell$-torsion in class groups, as
in \cite{Mas07, Mas07b}.  The resulting nonvanishing statements, however, are
typically ineffective.  Obtaining effective quantitative nonvanishing remains an
open problem.

By combining Corollaries \ref{cor1.4} and \ref{cor1.6} with subconvexity bounds
over number fields, such as \cite{Yan26b}, we obtain quantitative nonvanishing
for canonical Hecke $L$-functions over CM fields.  In particular, when
$d=[F:\mathbb{Q}]\geq 14$, the result is effective.

\begin{thmx}\label{thmc}
Let $0<\varepsilon<10^{-3}$ and suppose that
$(1+k+|t|)^{15d}N_F(\mathfrak{f})^{11}\ll D_{E/F}^{1-\varepsilon}L(1,\eta)^{16/11}$. Let $0\leq \vartheta\leq 7/64$ be a parameter towards the Ramanujan conjecture for unitary cuspidal automorphic representations of $\mathrm{GL}_2/F$. Then, for
$j\in \{0,1\}$, we have
\begin{align*}
\sum_{\chi\in X_{E/F}^{\mathrm{un}}(k,\boldsymbol{\Phi})}\mathbf{1}_{L^{(j)}(1/2+it,\widetilde{\mu}\chi)\neq 0}\gg_{F,\varepsilon} L(1,\eta)D_{E/F}^{\frac{1-2\vartheta}{12}-\varepsilon}N_F(\mathfrak{f})^{-\frac{1}{2}+\frac{1-2\vartheta}{12}-\varepsilon},
\end{align*}
with an effective implied constant. Consequently, the following two bounds hold.
\begin{itemize}
\item For any $d\geq 1$ and $(1+k+|t|)^{15d}N_F(\mathfrak{f})^{11}\ll D_{E/F}^{1-\varepsilon}$,
\begin{equation}\label{e1.5}
\sum_{\chi\in X_{E/F}^{\mathrm{un}}(k,\boldsymbol{\Phi})}\mathbf{1}_{L^{(j)}(1/2+it,\widetilde{\mu}\chi)\neq 0}\gg_{F,\varepsilon} D_{E/F}^{\frac{1-2\vartheta}{12}-\varepsilon}N_F(\mathfrak{f})^{-\frac{1}{2}+\frac{1-2\vartheta}{12}-\varepsilon},
\end{equation}
with an ineffective implied constant.
\item If $d\geq 14$ and $(1+k+|t|)^{15d}N_F(\mathfrak{f})^{11}\ll D_{E/F}^{1-\varepsilon}$, then
\begin{equation}\label{1.5}
\sum_{\chi\in X_{E/F}^{\mathrm{un}}(k,\boldsymbol{\Phi})}\mathbf{1}_{L^{(j)}(1/2+it,\widetilde{\mu}\chi)\neq 0}\gg_{F,\varepsilon} D_{E/F}^{\frac{1-2\vartheta}{12}-\frac{1}{2d}-\varepsilon}N_F(\mathfrak{f})^{-\frac{1}{2}+\frac{1-2\vartheta}{12}-\varepsilon},
\end{equation}
with an effective implied constant.
\end{itemize}
\end{thmx}

\begin{remark}
The ineffectivity in \eqref{e1.5} comes only from the possible exceptional zero
in the lower bound for $L(1,\eta)$. Using Stark's effective class-number lower
bound for CM fields \cite[Theorem 2]{Sta74}, this obstruction disappears when
$d\geq 14$, giving the effective estimate \eqref{1.5}. Improvements in effective
lower bounds for $L(1,\eta)$ would lower this threshold.
\end{remark}

\begin{remark}
For $j=0$, \eqref{e1.5} strengthens \cite[Theorem 1.2]{MY11} by allowing the
parameters to vary in the hybrid range
$(1+k+|t|)^{15d}N_F(\mathfrak{f})^{11}\ll D_{E/F}$, rather than keeping them
fixed. For $j=1$, Theorem \ref{thmc} is new beyond the classical case
$F=\mathbb{Q}$ and $\mu=\mathbf{1}$.
\end{remark}

\subsection{Ranks of Abelian Varieties}

By the theory of complex multiplication, a canonical Hecke character $\chi$ of
type $\Phi$ is associated with a CM abelian variety $A(\chi)$ over $E$, whose CM
field is generated by the values $\chi(\mathfrak{a})$, as $\mathfrak{a}$ ranges
over ideals of $E$ prime to the conductor of $\chi$. We denote by
$A(\chi)^{\mu}=A(\widetilde{\mu}\chi)$ the quadratic twist corresponding to the
twist of $\chi$ by $\mu$.

Combining \eqref{e1.5} in the case $j=k=0$ with the rank-zero theorem for the
associated CM abelian varieties used by Masri--Yang \cite{MY11}, and attributed
there to Tian--Zhang \cite{TZ08}, would give the existence of $\chi$ such that
$A(\chi)^{\mu}$ has Mordell--Weil rank zero. With the same BSD input for CM
abelian varieties, Theorem \ref{thmc} would therefore generalize
\cite[Theorem 1.3]{MY11}. However, the Tian--Zhang reference does not seem to
be available in published form. We therefore record instead a rank-one
consequence obtained by taking $j=1$ in Theorem \ref{thmc} and combining the
Gross--Zagier formula of Yuan--Zhang--Zhang \cite{YZZ13} with the Euler-system
method for CM points on Shimura curves, developed in particular by
Nekov'a\v{r} \cite{Nek07}.

\begin{thmx}\label{thmD}
Let $0<\varepsilon<10^{-3}$, and let $\mathfrak{f}\subseteq\mathcal{O}_F$ be
coprime to $2\mathfrak{D}_E$. Let $\mu$ be a quadratic character of conductor
$\mathfrak{f}$ with $\mu_{\infty}(-1)=-1$. Suppose that
$D_{E/F}\geq N_F(\mathfrak{f})^{11+\varepsilon}$. Then there are
\begin{equation}\label{1.7}
\gg_{F,\varepsilon}
D_{E/F}^{\frac{1-2\vartheta}{12}-\varepsilon}
N_F(\mathfrak{f})^{-\frac{1}{2}+\frac{1-2\vartheta}{12}-\varepsilon}
\end{equation}
characters $\chi\in X_{E/F}^{\mathrm{un}}(0,\boldsymbol{\Phi})$ such that
$A(\chi)^{\mu}$ has Mordell--Weil rank $1$ over $E$.
\end{thmx}

\begin{remark}
When $F=\mathbb{Q}$ and $\mu=\mathbf{1}$, Masri--Yang \cite{MY00} proved
Gross's conjecture, obtaining rank $1$ for all
$\chi\in X_{E/F}^{\mathrm{un}}(0,\boldsymbol{\Phi})$. For a nontrivial fixed
$\mu$-twist, or for $F\neq\mathbb{Q}$, Theorem \ref{thmD} appears to be new.
It gives a quantitative rank-one result, in contrast with the rank-zero result
of \cite[Theorem 1.3]{MY11}. In particular, it gives the explicit hybrid
threshold $D_{E/F}\geq N_F(\mathfrak{f})^{11+\varepsilon}$ for the existence
of twists $A(\chi)^{\mu}$ of Mordell--Weil rank $1$. Apart from its conductor
$\mathfrak{f}$, the bound \eqref{1.7} is uniform in the choice of $\mu$.
\end{remark}

For $k\geq 1$, the central values should instead be viewed as critical values
of higher-weight CM motives.  Their nonvanishing is naturally related to the
Bloch--Kato conjecture, which predicts the vanishing of the corresponding Selmer
groups and the absence of nontrivial motivic cycles in the relevant degree.
Thus the higher-weight case gives a motivic analogue of the rank-zero
phenomenon, rather than a direct Mordell--Weil rank statement.

\subsection{Further Remarks}

\begin{remark}[Comparison Principle]
Theorem \ref{thm3.1} is obtained by comparing Fourier expansions of an
Eisenstein series along the split torus and along the non-split torus
$E^{\times}\hookrightarrow \mathrm{GL}_2(F)$, using the same embedding
compatible with the inclusion $F\subset E$.  A higher-rank analogue, based on
comparing Fourier expansions associated with different embeddings
$\mathrm{GL}_2\hookrightarrow \mathrm{GL}_3$, was established in \cite{Yan25}.
Thus the present reciprocity formula compares different subgroups for a fixed
embedding, whereas the reciprocity formulas in loc. cit. compare different
embeddings of the same subgroup.
\end{remark}

\begin{remark}[Dyadic Places]
The assumption that every prime of $F$ above $2$ is unramified in $E$ is imposed
only to keep the canonical family in its simplest form.  At ramified dyadic
places, Rohrlich's local classification allows several conductor exponents.
One could remove this assumption by replacing, at such places $v$, the
characteristic function of $\mathcal{O}_{E_v}^{\times}$ by that of a suitable
order, without changing the spectral reciprocity mechanism.  We do not pursue
this additional bookkeeping here.
\end{remark}

\begin{remark}[Flexibility]
The spectral reciprocity formula in Theorem \ref{thm3.1} is flexible enough to
treat other arithmetic families.  For instance, taking
$\pi=|\cdot|^s\boxplus |\cdot|^{-s}$ and choosing suitable test functions
$h(\cdot,s)$ gives moment formulas for class group $L$-functions, yielding an
adelic proof of the first moment formula in \cite{Yan20}. Similar specializations apply to $L$-functions attached
to algebraic anticyclotomic Hecke characters.
\end{remark}

\subsection{Outline of the Paper}

We briefly describe the organization of the paper.

\begin{itemize}
\item In \textsection\ref{sec2} we recall the toric Fourier expansion of
Eisenstein series, and in \textsection\ref{section3} we prove the Fourier
spectral reciprocity formula, Theorem \ref{thm3.1}.

\item In \textsection\ref{sec3} we construct the arithmetic and automorphic data
used in the specialization to canonical Hecke characters.

\item In \textsection\ref{sect5}--\textsection\ref{sec7} we carry out the
explicit local and global computations: the toric integrals
$Z(\chi;h(\cdot,s))$, the degenerate terms
$J_{\mathrm{Const}}(h(\cdot,s))$ and $J_{\mathrm{Res}}(h(\cdot,s))$, and the
dual contribution $J_{\mathrm{Dual}}^{\heartsuit}(h(\cdot,s))$.

\item Finally, in \textsection\ref{sec8}, we combine these computations to prove
the first moment formula, Theorem \ref{thma}.
\end{itemize}

\subsection{Notation}\label{notation}

Let $F$ be a number field with ring of integers $\mathcal{O}_F$, adele ring
$\mathbb{A}_F$, absolute norm $N_F$, different $\mathfrak{D}_F$, and degree
$d=[F:\mathbb{Q}]$. Let $\Sigma_F$, $\Sigma_{F,\fin}$, and
$\Sigma_{F,\infty}$ denote respectively the sets of all, finite, and
Archimedean places of $F$.  For $v\in \Sigma_F$, let $F_v$ be the completion and
$|\cdot|_v$ the normalized absolute value.  If $v\in \Sigma_{F,\fin}$, write
$\mathcal{O}_v$ for the ring of integers, $\mathfrak{p}_v$ for the maximal
ideal, $q_v=\#(\mathcal{O}_v/\mathfrak{p}_v)$, and choose a uniformizer
$\varpi_v$.  Let $e_v$ be normalized by $e_v(\varpi_v)=1$, and set $d_{F_v}:=e_v(\mathfrak{D}_F)$. 

For an integral ideal $\mathfrak{n}\subseteq \mathcal{O}_F$, we write
$v\mid\mathfrak{n}$ if $\mathfrak{n}\subseteq \mathfrak{p}_v$, and $N_F(\mathfrak{n})=\#(\mathcal{O}_F/\mathfrak{n})$. 
On $\mathbb{A}_F^{\times}$ we put $|\cdot|=\prod_{v\in \Sigma_F}|\cdot|_v$.

The finite and completed Dedekind zeta functions are denoted by
\begin{align*}
\zeta_F(s)=\prod_{v<\infty}\zeta_{F_v}(s),\quad
\Lambda_F(s)=\prod_{v\in \Sigma_F} \zeta_{F_v}(s),
\end{align*}
and are always understood by meromorphic continuation.  We use the standard
additive character $\psi=\psi_{\mathbb Q}\circ \Tr_F$ of $F\backslash\mathbb A_F$, where $\psi_{\mathbb Q}(x)=e^{2\pi i x}$ on
$\mathbb R$.  Let $dt_v$ be the self-dual measure with respect to $\psi_v$, and
put $dt=\prod_vdt_v$.  Multiplicative measures are normalized by $d^{\times}t_v=|t_v|_v^{-1}dt_v$ if $v\in \Sigma_{F,\infty}$, and $d^{\times}t_v=\zeta_{F_v}(1)|t_v|_v^{-1}dt_v$ if $v\in \Sigma_{F,\fin}$. Thus, for $v\in \Sigma_{F,\fin}$, $\Vol(\mathcal{O}_v^{\times},d^{\times}t_v)=q_v^{-d_{F_v}/2}$, and globally
\begin{align*}
\Vol(F\backslash\mathbb A_F,dt)=1,\quad
\Vol(F^{\times}\backslash\mathbb A_F^{(1)},d^{\times}t)
=\underset{s=1}{\Res}\ \zeta_F(s).
\end{align*}
Here $\mathbb A_F^{(1)}$ is the subgroup of ideles of norm $1$, and we write $\widehat{C}_F:=\widehat{F^{\times}\backslash\mathbb A_F^{(1)}}$. 

Let $E/F$ be a quadratic extension, and let $\eta=\eta_{E/F}$ be the associated
quadratic character.  We use the same conventions over $E$; in particular, for
$w\in \Sigma_{E,\fin}$, $\mathcal{O}_{E_w}$ denotes the ring of integers of $E_w$, and
$d_{E_w}:=e_w(\mathfrak{D}_E)$.

Let $G=\mathrm{GL}_2$, let $N$ be the 
unipotent radical of the Borel subgroup. We take $K_v=G(\mathcal{O}_v)$ if $v\in \Sigma_{F,\fin}$. 

For an automorphic representation $\pi=\otimes_v'\pi_v$, let $L(s,\pi)$ denote
the finite $L$-function, $L_{\infty}(s,\pi)$ its Archimedean factor, and $\Lambda(s,\pi)=L_{\infty}(s,\pi)L(s,\pi)$. 

For $\eta_1,\eta_2\in \widehat{C}_F$, write $\eta_1\boxplus \eta_2:=\Ind(\eta_1\otimes\eta_2)$ 
for the corresponding induced representation.  The standard generic
character of $[N]$ is
\begin{align*}
\theta\left(\begin{pmatrix}
1 & b\\
& 1
\end{pmatrix}\right)=\psi(b),\quad b\in F\backslash\mathbb A_F.
\end{align*}

Finally, set
\begin{equation}\label{eq1.1}
\Gamma_{\mathbb{R}}(s):=\pi^{-s/2}\Gamma(s/2),\quad
\Gamma_{\mathbb{C}}(s):=2(2\pi)^{-s}\Gamma(s).
\end{equation}
Throughout, $\varepsilon>0$ denotes an arbitrarily small constant which may vary
from line to line.

\section{Fourier Expansions over Torus}\label{sec2}
\subsection{Eisenstein series}
Let $\mu$ and $\omega$ be characters of
$F^{\times}\backslash\mathbb{A}_F^{\times}$, and let $s\in\mathbb{C}$. Denote by $\widetilde{\mu}:=\mu\circ\mathrm{Nr}_{E/F}$. 
\subsubsection{Eisenstein series}
Let $h(\cdot,s)$ be a section in the induced representation $\mu|\cdot|^s\boxplus \overline{\mu}\omega|\cdot|^{-s}$. For
$g\in G(\mathbb{A}_F)$, define
\begin{equation}\label{2.1}
E(g;h(\cdot,s))
:=
\sum_{\delta\in B(F)\backslash G(F)}
h(\delta g,s).
\end{equation}
Then $E(g;h(\cdot,s))$ converges absolutely for
$\Re(s)\gg 1$ and admits meromorphic continuation to $s\in\mathbb{C}$. 
For brevity, we write $E(g,s)$ for $E(g;h(\cdot,s))$; henceforth, $E(g,s)$ is identified with its meromorphic continuation. 

\subsubsection{Godement Section}
Let
$\Phi(\cdot,\cdot)$ be a suitable function on $\mathbb{A}_F^2$. We construct the section $h(\cdot,s)$ explicitly in the following Godement form, as a Tate-type integral:
\begin{equation}\label{e2.2}
h(g,s):=\mu(\det g)|\det g|^{1/2+s}
\int_{\mathbb{A}_F^{\times}}
\Phi((0,t)g)\mu^2\overline{\omega}(t)|t|^{1+2s}d^{\times}t. 
\end{equation}

\subsubsection{Fourier Expansion}\label{sec2.1.2}
Let
\begin{align*}
W(g,s):=\int_{[N]}E(ug,s)\overline{\theta}(u)du
\end{align*}
be the Whittaker function attached to $E$. Explicitly, in the form \eqref{e2.2}, we have 
\begin{equation}\label{2.2}
W(g,s)=\mu(\det g)|\det g|^{1/2+s}
\int_{\mathbb{A}_F^{\times}}\int_{\mathbb{A}_F}
\Phi((t,ct)g)\overline{\psi}(c)dc
\mu^{2}\overline{\omega}(t)|t|^{1+2s}d^{\times}t.
\end{equation}

The Fourier expansion of $E$ takes the form
\begin{equation}\label{eq2.6}
E(g,s)=h(g,s)+h^{\Diamond}(g,s)+\mathcal{G}(g,s),
\end{equation}
where
\begin{align*}
h^{\Diamond}(g,s):=\int_{N(\mathbb{A}_F)}h(wug,s)du,\quad 
\mathcal{G}(g,s):=\sum_{\alpha\in F^{\times}}
W\left(\begin{pmatrix}
\alpha\\
& 1
\end{pmatrix}g,s\right).
\end{align*}

\subsection{Toric Period Integrals}\label{sec2.5}
\subsubsection{Split Period Integrals}\label{sec2.3.1}
For $\Re(s)\gg 1$ and $\xi\in \widehat{C}_F$, define 
\begin{equation}\label{2.5}
\Psi(W(\cdot,s),\xi|\cdot|^{\lambda}):=\int_{\mathbb{A}_F^{\times}}
W\left(\begin{pmatrix}
y\\
& 1
\end{pmatrix},s\right)\xi(y)|y|^{\lambda}d^{\times}y.
\end{equation}
This integral converges absolutely when $\Re(\lambda)-|\Re(s)|$ is sufficiently large, and satisfies 
\begin{equation}\label{e2.6}
\Psi(W(\cdot,s),\xi|\cdot|^{\lambda})\propto \Lambda(1/2+\lambda+s,\mu\xi)\Lambda(1/2+\lambda-s,\overline{\mu}\omega\xi).
\end{equation}
It admits a meromorphic continuation $\widetilde{\Psi}(W(\cdot,s),\xi|\cdot|^{\lambda})$ to $(s,\lambda)\in \mathbb{C}^2$.

\subsubsection{Non-split Period Integrals}
Let $E/F$ be a quadratic extension of number fields. Let
$\mathcal{C}_{E/F}=E^{\times}\mathbb{A}_F^{\times}\backslash\mathbb{A}_E^{\times}$,
which is compact, and let $\widehat{\mathcal{C}}_{E/F}$ denote its
Pontryagin dual. Write $E=F(\tau)$,
where $\tau$ is a fixed nonsquare in $F^{\times}$. Fix $\theta\in E^{\times}$
such that $\theta^2=\tau$, and consider the embedding
\begin{equation}\label{equ2.4}
\mathbb{A}_E^{\times}\hookrightarrow M_{2\times 2}(\mathbb{A}_F),\quad  x=a+b\theta\mapsto \iota(x):=\begin{pmatrix}
a & b\tau \\
b & a
\end{pmatrix},\ \ a, b\in \mathbb{A}_F.
\end{equation}

For $\Re(s)\gg 1$ and $\chi\in \widehat{C}_E$, we define, for a suitable function $\Phi(\cdot,\cdot)$ on $\mathbb{A}_F^2$, the zeta integral
\begin{align*}
Z(s,\Phi,\chi)=\int_{\mathbb{A}_E^{\times}}
\Phi((0,1)\iota(x))\chi(x)|x|_{\mathbb{A}_E}^{1/2+s}d^{\times}x.
\end{align*}
By Tate's thesis, $Z(s,\Phi,\chi)$ converges absolutely for $\Re(s)\gg 1$, admits meromorphic continuation, and satisfies
\begin{align*}
Z(s,\Phi,\chi)\propto \Lambda(1/2+s,\chi). 
\end{align*}
Henceforth, we use the same notation $Z(s,\Phi,\chi)$ for this meromorphic continuation.

If $\overline{\omega}=\chi|_{\mathbb{A}_F^{\times}}$, define
\begin{align*}
Z(\chi;h(\cdot,s)):=\int_{\mathbb{A}_F^{\times}\backslash\mathbb{A}_E^{\times}}h(\iota(x),s)\chi(x)d^{\times}x. 
\end{align*}
The following lemma relates this period integral to the Tate zeta integral above.

\begin{lemma}\label{lem2.1}
Let $\chi\in \widehat{C}_E$ be such that
$\overline{\omega}=\chi|_{\mathbb{A}_F^{\times}}$. We have
\begin{equation}\label{2.6}
\int_{\mathcal{C}_{E/F}}E(\iota(x),s)
\chi(x)d^{\times}x=Z(\chi;h(\cdot,s)).
\end{equation}
Moreover, if $h(\cdot,s)$ is of the form \eqref{e2.2}, then 
\begin{equation}\label{2.7}
Z(\chi;h(\cdot,s))=Z(s,\Phi,\widetilde{\mu}\chi).
\end{equation}
\end{lemma}
\begin{proof}
Define
\begin{equation}\label{2.8}
P(s,\chi) :=\int_{\mathcal{C}_{E/F}}E(\iota(x),s)
\chi(x)d^{\times}x.
\end{equation}
Since $\mathcal{C}_{E/F}$ is compact, $P(s,\chi) $ is meromorphic, with its analytic behavior inherited from that of $E(\iota(x),s)$.

Suppose $\Re(s)\gg 1$. Since $G(F)=B(F)\iota(E^{\times})$ and
$B(F)\cap \iota(E^{\times})\simeq F^{\times}$, it follows from \eqref{2.1} that
\begin{equation}\label{eq2.4}
P(s,\chi) =\int_{E^{\times}\mathbb{A}_F^{\times}\backslash\mathbb{A}_E^{\times}}\sum_{\delta\in F^{\times}\backslash E^{\times}}h(\iota(\delta x),s)\chi(x)d^{\times}x.
\end{equation}

Changing variables and using the assumption
$\overline{\omega}=\chi|_{\mathbb{A}_F^{\times}}$, we obtain \eqref{2.6} from \eqref{eq2.4}. If $h(\cdot,s)$ is of the form \eqref{e2.2}, then \eqref{2.6} gives 
\begin{align*}
Z(\chi;h(\cdot,s))=\int_{\mathbb{A}_E^{\times}}
\Phi((0,1)\iota(x))\mu(\det \iota(x))\chi(x)|\det \iota(x)|_{\mathbb{A}_F}^{1/2+s}d^{\times}x.	
\end{align*}
Using $\det \iota(x)=\mathrm{Nr}_{E/F}(x)$, this becomes
\begin{align*}
Z(\chi;h(\cdot,s))=\int_{\mathbb{A}_E^{\times}}
\Phi((0,1)\iota(x))\widetilde{\mu}\chi(x)|x|_{\mathbb{A}_E}^{1/2+s}d^{\times}x=Z(s,\Phi,\widetilde{\mu}\chi).
\end{align*}
This proves \eqref{2.7}.
\end{proof}

\section{A Spectral Reciprocity of Fourier Type}\label{section3}
Let $E/F$ be a quadratic extension of number fields, and let $\eta$ be the quadratic character associated with $E/F$. 

\SR*

We refer to $J_{\mathrm{Const}}(h(\cdot,s))$ as the \textit{constant term}, since it arises from the constant terms of Eisenstein series. The term $J_{\mathrm{Dual}}(h(\cdot,s))$ is called the \textit{dual side}. Finally, $J_{\mathrm{Res}}(h(\cdot,s))$ is called the \textit{residual term}, since it comes from the residues of the spectral periods on the dual side.

\subsection{A Coarse Spectral Reciprocity}
\begin{prop} \label{prop2.2}
Let $s\in \mathbb{C}$ and $r>10+|\Re(s)|$. We have 
\begin{align*}
\sum_{\chi\in \widehat{C}_E:\
\chi|_{\mathbb{A}_F^{\times}}=\overline{\omega}}\frac{Z(\chi;h(\cdot,s))}{L(1,\eta)}=h(I_2,s)+h^{\Diamond}(I_2,s)
+J_{\mathrm{Dual}}(h(\cdot,s)),
\end{align*}
where 
\begin{align*}
J_{\mathrm{Dual}}(h(\cdot,s)):=\frac{1}{\underset{s=1}{\Res}\ \zeta_F(s)}\sum_{\xi\in \widehat{C}_F}\frac{1}{2\pi i}\int_{(r)}\Psi(W(\cdot,s),\xi|\cdot|^{\lambda})d\lambda.
\end{align*}
Here $\int_{(r)}$ means the integral is over the vertical path $r+i\mathbb{R}$. 
\end{prop}
\begin{proof}
Fix one Hecke character $\chi_0$ of $E^{\times}\backslash\mathbb A_E^{\times}$
such that $\chi_0|_{\mathbb A_F^{\times}}=\overline{\omega}$. Then
$x\mapsto E(\iota(x),s)\chi_0(x)$ is well-defined on $\mathcal C_{E/F}$.
Fourier inversion on the compact group $\mathcal C_{E/F}$ gives
\begin{align*}
E(I_2,s)=\Vol(\mathcal C_{E/F})^{-1}
\sum_{\rho\in \widehat{\mathcal C}_{E/F}}
\int_{\mathcal C_{E/F}}
E(\iota(x),s)\chi_0(x)\rho(x)d^{\times} x.
\end{align*}
Equivalently, after writing $\chi=\chi_0\rho$, this becomes
\begin{equation}\label{2.9}
E(I_2,s)=\Vol(\mathcal{C}_{E/F})^{-1}\sum_{\chi\in \widehat{C}_E:\
\chi|_{\mathbb{A}_F^{\times}}=\overline{\omega}}\int_{\mathcal{C}_{E/F}}E(\iota(x),s)
\chi(x)d^{\times}x.
\end{equation}

By definition,
\begin{align*}
\mathcal{C}_{E/F}\simeq (E^{\times}\backslash\mathbb{A}_E^{(1)})/(F^{\times}\backslash\mathbb{A}_F^{(1)}).
\end{align*}
Hence
\begin{equation}\label{2.10}
\Vol(\mathcal{C}_{E/F})=\frac{\Vol(E^{\times}\backslash\mathbb{A}_E^{(1)})}{\Vol(F^{\times}\backslash\mathbb{A}_F^{(1)})}=\frac{\underset{s=1}{\Res}\ \zeta_E(s)}{\underset{s=1}{\Res}\ \zeta_F(s)}=L(1,\eta).
\end{equation}

Substituting \eqref{2.10} into \eqref{2.9}, and using Lemma \ref{lem2.1}, we obtain
\begin{equation}\label{2.11}
E(I_2,s)=L(1,\eta)^{-1}\sum_{\chi\in \widehat{C}_E:\
\chi|_{\mathbb{A}_F^{\times}}=\overline{\omega}}Z(\chi;h(\cdot,s)). 
\end{equation}

For $y\in \mathbb{A}_F^{\times}$, consider the auxiliary function
\begin{align*}
F(y):=|y|^{r}\sum_{\alpha\in F^{\times}}
W\left(\begin{pmatrix}
\alpha y\\
& 1
\end{pmatrix},s\right).
\end{align*}

Since $r>10+|\Re(s)|$, the standard bounds for Whittaker functions, for example \cite[Proposition 3.2.3]{MV10}, imply that
\begin{align*}
F(y)\in L^1(F^{\times}\backslash\mathbb{A}_F^{\times})\cap L^2(F^{\times}\backslash\mathbb{A}_F^{\times}).
\end{align*}
We may therefore expand $F(\cdot)$ along the spectrum of $F^{\times}\backslash\mathbb{A}_F^{\times}$, obtaining
\begin{equation}\label{2.12}
F(y)=\frac{1}{2\pi\Vol(F^{\times}\backslash\mathbb{A}_F^{(1)})}\sum_{\xi\in \widehat{C}_F}\int_{\mathbb{R}}\int_{F^{\times}\backslash\mathbb{A}_F^{\times}}F(a)\xi(a)|a|^{it}d^{\times}a\, \overline{\xi}(y)|y|^{-it} dt.
\end{equation}

By definition \eqref{2.5}, we have
\begin{align*}
\int_{F^{\times}\backslash\mathbb{A}_F^{\times}}F(a)\xi(a)|a|^{it}d^{\times}a=\Psi(W(\cdot,s),\xi|\cdot|^{r+it}).
\end{align*}
Taking $y=I_2$ in \eqref{2.12}, we obtain
\begin{equation}\label{2.14}
\mathcal{G}_2(I_2,s)=F(1)=\frac{1}{\underset{s=1}{\Res}\ \zeta_F(s)}\sum_{\xi\in \widehat{C}_F}\frac{1}{2\pi i}\int_{(r)}\Psi(W(\cdot,s),\xi|\cdot|^{\lambda})d\lambda.
\end{equation}

Proposition \ref{prop2.2} now follows from \eqref{2.11}, \eqref{2.14}, and the Fourier expansion \eqref{eq2.6}. 
\end{proof}

\subsection{Meromorphic Continuation}
Let $|\Re(s)|<1/2$. By \eqref{e2.6} and Cauchy's formula, we obtain 
\begin{multline}\label{3.7}
\frac{1}{2\pi i}\int_{(r)}\Psi(W(\cdot,s),\xi|\cdot|^{\lambda})d\lambda=\frac{1}{2\pi i}\int_{(0)}\widetilde{\Psi}(W(\cdot,s),\xi|\cdot|^{\lambda})d\lambda\\
+\mathbf{1}_{\xi=\overline{\mu}}\underset{\lambda=1/2-s}{\Res}\,\widetilde{\Psi}(W(\cdot,s),\xi|\cdot|^{\lambda})+\mathbf{1}_{\xi=\mu\overline{\omega}}\underset{\lambda=1/2+s}{\Res}\,\widetilde{\Psi}(W(\cdot,s),\xi|\cdot|^{\lambda}).
\end{multline}

The functions $\underset{\lambda=1/2\pm s}{\Res}\,\widetilde{\Psi}(W(\cdot,s),\xi|\cdot|^{\lambda})$ are meromorphic.  
\begin{lemma}\label{lemma3.3}
We have the following identities of meromorphic functions:
\begin{equation}\label{f3.8}
\underset{\lambda=1/2-s}{\Res}\,\widetilde{\Psi}(W(\cdot,s),\overline{\mu}|\cdot|^{\lambda})=h^{\Diamond}(w,s)\cdot \underset{\lambda=1}{\Res}\,\Lambda_F(\lambda),
\end{equation}
and
\begin{equation}\label{f3.9}
\underset{\lambda=1/2+s}{\Res}\,\widetilde{\Psi}(W(\cdot,s),\mu\overline{\omega}|\cdot|^{\lambda})=h(w,s)\cdot \underset{\lambda=1}{\Res}\,\Lambda_F(\lambda).
\end{equation}
\end{lemma}
\begin{proof}
Suppose that $\Re(\lambda)>1/2+|\Re(s)|$ and that $h(\cdot,s)$ is the Godement section \eqref{2.2}.

By a change of variables, the function $\Psi(W(\cdot,s),\xi|\cdot|^{\lambda})$ is represented by 
\begin{align*}
\int_{\mathbb{A}_F^{\times}}
\int_{\mathbb{A}_F^{\times}}\int_{\mathbb{A}_F}
\Phi(y,c)\overline{\psi}(ct)dc
\xi\overline{\mu}\omega(t)|t|^{1/2-s+\lambda}d^{\times}t\xi\mu(y)|y|^{1/2+s+\lambda}d^{\times}y.
\end{align*}

Consequently, when $\Re(s)<0$, Tate's thesis gives
\begin{equation}\label{e3.10}
\underset{\lambda=\frac{1}{2}-s}{\Res}\,\widetilde{\Psi}(W(\cdot,s),\overline{\mu}|\cdot|^{\lambda})=\beta
\int_{\mathbb{A}_F^{\times}}\int_{\mathbb{A}_F}\int_{\mathbb{A}_F}
\Phi(b,c)db\overline{\psi}(ct)dc
\overline{\mu}^2\omega(t)|t|^{1-2s}d^{\times}t,
\end{equation}
where $\beta:=\underset{\lambda=1}{\Res}\,\Lambda_F(\lambda)$.  

The right-hand side of \eqref{e3.10} is a Tate integral representing $\Lambda(1-2s,\overline{\mu}^2\omega)$. Hence, by meromorphic continuation and the functional equation, when $\Re(s)\gg 1$ we have
\begin{align*}
\underset{\lambda=1/2-s}{\Res}\,\widetilde{\Psi}(W(\cdot,s),\mu|\cdot|^{\lambda})=\underset{\lambda=1}{\Res}\,\Lambda_F(\lambda)\int_{\mathbb{A}_F^{\times}}\int_{\mathbb{A}_F}\Phi(b,t)db
\mu^2\overline{\omega}(t)|t|^{2s}d^{\times}t,
\end{align*}
which can be rewritten as
\begin{multline}\label{e3.11}
\underset{\lambda=1/2-s}{\Res}\,\widetilde{\Psi}(W(\cdot,s),\mu|\cdot|^{\lambda})=\int_{\mathbb{A}_F^{\times}}\int_{\mathbb{A}_F}\Phi\left((0,t)w\begin{pmatrix}
1 & b\\
& 1
\end{pmatrix}w\right)db\\
\mu^2\overline{\omega}(t)|t|^{1+2s}d^{\times}t\, \underset{\lambda=1}{\Res}\,\Lambda_F(\lambda).
\end{multline}

Therefore, \eqref{f3.8} follows from \eqref{e2.2}, \eqref{e3.11}, and meromorphic continuation. Varying $\Phi$, we obtain \eqref{f3.8} for an arbitrary section $h(\cdot,s)$. The proof of \eqref{f3.9} is identical.
\end{proof}

\subsection{Proof of Theorem \ref{thm3.1}}
Theorem \ref{thm3.1} follows readily from Proposition \ref{prop2.2}, \eqref{3.7}, and Lemma \ref{lemma3.3}.

\section{Explicit Arithmetic and Automorphic Data}\label{sec3}
\subsection{CM Fields}
Let $F$ be a totally real field and let $E/F$ be a CM extension with CM type $\boldsymbol{\Phi}$. Let $\eta$ be the quadratic character associated with $E/F$. 

Recall that $\widehat{\mathcal{C}}_{E/F}$ is the Pontryagin dual of $\mathcal{C}_{E/F}=E^{\times}\mathbb{A}_F^{\times}\backslash\mathbb{A}_E^{\times}$. Let $\widehat{\mathcal{C}}_{E/F}^{\text{un}}$ be the subgroup of $\widehat{\mathcal{C}}_{E/F}$ consisting of everywhere unramified characters. Let  
\begin{align*}
\mathrm{Cl}_{E/F}:=\mathrm{Cl}_{E}/\Im(\mathrm{Cl}_{F})
\end{align*}
be the relative ideal class group.

\subsection{Construction of Relevant Data}\label{sect3.1}
Throughout this paper, we assume that every prime of $F$ above $2$ is unramified in $E$. This mild hypothesis merely removes dyadic exceptional cases and allows the local formulas to be stated uniformly; it does not conceal any essential analytic difficulty. Under this assumption, all finite ramification in $E/F$ is tame, and the construction of canonical Hecke characters takes a particularly simple form. 
 
\subsubsection{Canonical Hecke Characters}
Let $k\in \mathbb{Z}_{\geq 0}$. Let $\chi$ be a canonical Hecke character of $E$ of type $(2k+1)\boldsymbol{\Phi}$, namely an algebraic Hecke character satisfying the following three conditions:
\begin{enumerate}
\item[(a).] $\chi$ is ramified precisely at the primes dividing the relative different $\mathfrak{D}_{E/F}$; 
\item[(b).] $\chi(\overline{\mathfrak{A}})=\overline{\chi(\mathfrak{A})}$ for all ideals $\mathfrak{A}$ of $E$ prime to $\mathfrak{D}_{E/F}$;
\item[(c).] $\chi(\alpha\mathcal{O}_E)=\pm\prod_{\sigma\in \boldsymbol{\Phi}}\sigma(\alpha)^{2k+1}$ for $\alpha\in E^{\times}$ prime to $\mathfrak{D}_{E/F}$. 
\end{enumerate} 

Let $\chi^{\mathrm{un}}:=\chi|\cdot|_{\mathbb{A}_E}^{-k-1/2}$ be the unitary normalization of $\chi$. Then for $v\in \Sigma_{F,\infty}$, we have $E_v\simeq \mathbb{C}$ via the embedding $\sigma_v\in \boldsymbol{\Phi}$, and 
\begin{equation}\label{e3.1}
\chi_v^{\mathrm{un}}(z)=\left(\frac{z}{|z|}\right)^{2k+1},\quad z\in \mathbb{C}^{\times}. 
\end{equation}
If instead we use the conjugate embedding, then $\chi_v^{\mathrm{un}}(z)$ has the same form as in \eqref{e3.1}, with $2k+1$ replaced by $-(2k+1)$. 

By \cite[Proposition 1]{Roh82}, condition (b) characterizing $\chi$ is equivalent to 
\begin{equation}\label{e3.2}
\chi^{\mathrm{un}}|_{\mathbb{A}_F^{\times}}=\eta. 
\end{equation}
Moreover, since every prime of $F$ above $2$ is unramified in $E$, the relative different $\mathfrak{D}_{E/F}$ is squarefree. Hence condition (a) is equivalent to the assertion that the arithmetic conductor of $\chi^{\mathrm{un}}$ is precisely $\mathfrak{D}_{E/F}$. 

Let $X_{E/F}^{\mathrm{un}}(k,\boldsymbol{\Phi})$ be the set of unitary canonical Hecke characters of $E$ of type $(2k+1)\boldsymbol{\Phi}$. This set was determined by Rohrlich \cite{Roh82} for $k=0$, and the same argument applies for general $k$. In particular, for any fixed $\chi^*\in X_{E/F}^{\mathrm{un}}(k,\boldsymbol{\Phi})$, we have
\begin{align*}
X_{E/F}^{\mathrm{un}}(k,\boldsymbol{\Phi})=\Big\{\chi^*\chi:\ \chi\in \widehat{\mathrm{Cl}}_{E/F}\Big\}.
\end{align*}
Here $\widehat{\mathrm{Cl}}_{E/F}$ is the Pontryagin dual of $\mathrm{Cl}_{E/F}$.

\subsubsection{Automorphic Data}\label{sec4.2.2}
Let $\omega=\eta$, and let $\mu$ be a unitary quadratic Hecke character of $F^{\times}\backslash\mathbb{A}_F^{\times}$. Suppose $\mu_{\infty}(-1)=(-1)^k$, and that the arithmetic conductor $\mathfrak{f}$ of $\mu$ is prime to $2\mathfrak{D}_E$. 

We consider the Eisenstein series $E(\cdot,s)$ associated with the induced representation $\mu\boxplus \overline{\mu}\omega$, constructed as in \eqref{2.1}, with the section $h(\cdot,s)$ constructed in the next subsection.

\subsubsection{Arithmetic Data}
Let $\mathrm{St}(E/F)$ denote the Steinitz class of $\mathcal{O}_E$ over $\mathcal{O}_F$. Let $\tau\in \mathcal{O}_F$ be such that
\begin{equation}\label{4.3}
E=F(\sqrt{\tau}),\quad (\tau)=\mathfrak{D}_{E/F}\mathfrak{a}^2
\end{equation}
for some integral ideal $\mathfrak{a}\in \mathrm{St}(E/F)^{-1}$ satisfying $\mathfrak{a}+\mathfrak{f}\mathfrak{D}_{E}=\mathcal{O}_F$. We henceforth fix such a choice of $\tau$ throughout the paper. 

\subsection{Construction of the Section}\label{sec3.1.3}
For simplicity, write $\eta=\otimes_{v\in \Sigma_F}'\eta_v$. We construct $h(\cdot,s)=\otimes_{v\in \Sigma_F}'h_v(\cdot,s)$ as the Godement section \eqref{e2.2} associated with a suitable test function $\Phi(\cdot,\cdot)=\otimes_{v\in \Sigma_F}\Phi_v(\cdot,\cdot)$, namely
\begin{equation}\label{f3.3}
h_v\left(g_v
,s\right):=\mu_v(\det g_v)|\det g_v|_v^{1/2+s}
\int_{F_v^{\times}}
\Phi_v((0,t_v)g_v)\eta_v(t_v)|t_v|_v^{1+2s}d^{\times}t_v.
\end{equation} 

For $v\in \Sigma_{F}$ we define $\Phi_v(\cdot,\cdot)$ as follows.  
\begin{itemize}
\item at $v\in \Sigma_{F,\infty}$, let
\begin{align*}
\Phi_v(b_v,a_v):=(a_v-b_v\sqrt{\tau_v})^{2k+1}e^{-2\pi(a_v^2+b_v^2|\tau_v|_v)},\quad a_v, b_v\in F_v\simeq \mathbb{R}.
\end{align*}
\item at $v\in \Sigma_{F,\fin}$ with $v\mid \mathfrak{D}_{E}$, let
\begin{align*}
\Phi_v(b_v,a_v)=q_v^{d_{E_v}/2}\eta_v(a_v)\mathbf{1}_{\mathcal{O}_v^{\times}}(a_v)\mathbf{1}_{\mathcal{O}_v}(b_v),\quad a_v, b_v\in F_v. 
\end{align*}

\item at $v\in \Sigma_{F,\fin}$ with $v\mid \mathfrak{f}$, let 
\begin{align*}
\Phi_v(b_v,a_v)=\mu_v(a_v^2-b_v^2\tau_v)\mathbf{1}_{\mathcal{O}_v}(a_v)\mathbf{1}_{\mathcal{O}_v}(b_v)\mathbf{1}_{\mathcal{O}_v^{\times}}(a_v^2-b_v^2\tau_v),\quad a_v, b_v\in F_v. 
\end{align*}

\item at $v\in \Sigma_{F,\fin}$ with $v\nmid \mathfrak{f}\mathfrak{D}_{E}$, let $\mathfrak{a}_v=\mathfrak{p}_v^{e_v(\mathfrak{a})}$ and set 
\begin{align*}
\Phi_v(b_v,a_v)=q_v^{d_{E_v}/2}\mathbf{1}_{\mathcal{O}_v}(a_v)\mathbf{1}_{\mathfrak{a}_v^{-1}}(b_v),\quad a_v, b_v\in F_v. 
\end{align*}
\end{itemize}

\section{Explicit Formula for $Z(s,\Phi,\widetilde{\mu}\chi^*\chi)$}\label{sect5}
Let $\mu$ and $\Phi$ be constructed as in \textsection\ref{sec3}. The main result of this section is the following explicit formula.

\begin{prop}\label{prop3.2}
Let $\chi^*\in X_{E/F}^{\mathrm{un}}(k,\boldsymbol{\Phi})$ and $\chi\in \widehat{\mathcal{C}}_{E/F}$. Then, for $s\in \mathbb{C}$, we have an identity of meromorphic functions:
\begin{equation}\label{e3.3}
Z(s,\Phi,\widetilde{\mu}\chi^*\chi)=\frac{\pi^{d}\Gamma_{\mathbb{C}}(s+k+1)^{d}}{2^{d}}L(1/2+s,\widetilde{\mu}\chi^*\chi)\mathbf{1}_{\chi\in \widehat{\mathcal{C}}_{E/F}^{\text{un}}}.
\end{equation}
\end{prop}
\begin{proof}
It suffices to prove the identity for $\Re(s)\gg 1$. By definition,
\begin{equation}\label{e3.4}
Z(s,\Phi,\widetilde{\mu}\chi^*\chi)=\prod_{v\in \Sigma_F}Z_v(s,\Phi_v,\widetilde{\mu}_v\chi_v^*\chi_v),
\end{equation}
where
\begin{align*}
Z_v(s,\Phi_v,\widetilde{\mu}_v\chi_v^*\chi_v):=\int_{E_v^{\times}}
\Phi_v((0,1)\iota(x_v))\widetilde{\mu}_v\chi_v^*\chi_v(x_v)|\mathrm{Nr}_{E_v/F_v}(x_v)|_v^{1/2+s}d^{\times}x_v.
\end{align*}

For $v\in \Sigma_F$, write $x_v=a_v+b_v\sqrt{\tau_v}$. Then
\begin{align*}
\iota(x_v)=\begin{pmatrix}
a_v & b_v\tau_v\\
b_v & a_v
\end{pmatrix}.
\end{align*} 
We now compute the local integrals according to the definition of $\Phi_v(\cdot,\cdot)$ in \textsection\ref{sec3.1.3}.

\begin{itemize}
\item Suppose $v\in \Sigma_{F,\infty}$. Identifying $E_v$ with $\mathbb{C}$ via the embedding $\sigma_v\in \boldsymbol{\Phi}$, we have 
\begin{equation}\label{3.2}
\Phi_v((0,1)\iota(x_v))=\Phi_v(b_v,a_v)=\overline{z}^{2k+1}e^{-2\pi|z|^2}\mathbf{1}_{\Im(z)\geq 0},\quad z=\sigma_v(x_v),
\end{equation}
Here $|\cdot|$ denotes the usual Euclidean norm.

Since $\mu_v$ is quadratic, we have $\mu_v=\sgn^{i}$ for some $i\in \{0,1\}$, and hence
\begin{equation}\label{3.3}
\widetilde{\mu}_v\chi_v^*\chi_v(x_v)=\chi_v^*\chi_v(x_v),\quad x_v\in E_v^{\times}. 
\end{equation}
By definition, $\chi_v^*$ satisfies \eqref{e3.1}. Thus \eqref{3.2} and \eqref{3.3} give
\begin{align*}
Z_v(s,\Phi_v,\widetilde{\mu}_v\chi_v^*\chi_v)=\int_{\mathbb{C}^{\times}}
|z|^{2k+1}e^{-2\pi|z|^2}\chi_v(z)|z|^{1+2s}d^{\times}z.
\end{align*}

In polar coordinates $z=re^{i\theta}$, this becomes
\begin{equation}\label{e4.5}
Z_v(s,\Phi_v,\widetilde{\mu}_v\chi_v^*\chi_v)=\int_0^{2\pi}\int_0^{\infty}e^{-2\pi r^2}\chi_v(z)r^{2k+1+2s}drd\theta.
\end{equation}

Since $E_v/F_v$ is compact, we have $\chi_v(z)=(z/|z|)^m=e^{im\theta}$ for some $m\in \mathbb{Z}$.  Moreover,
\begin{align*}
\int_0^{2\pi} e^{im\theta}d\theta
=
\begin{cases}
2\pi, & m=0,\\
0, & m\in \mathbb Z,\ m\neq 0.
\end{cases}
\end{align*}

It follows from \eqref{e4.5} that
\begin{equation}\label{3.4}
Z_v(s,\Phi_v,\widetilde{\mu}_v\chi_v^*\chi_v)=\frac{\pi\Gamma_{\mathbb{C}}(s+k+1)\mathbf{1}_{m=0}}{2}=\frac{\pi \Gamma_{\mathbb{C}}(s+k+1)\mathbf{1}_{\text{$\chi_v$ unramified}}}{2}.
\end{equation}

\item Suppose $v\in \Sigma_{F,\fin}$ with $v\mid \mathfrak{D}_{E}$. Then
\begin{align*}
\Phi_v((0,1)\iota(x_v))=\Phi_v(b_v,a_v)=q_v^{d_{E_v}/2}\eta_v(a_v)\mathbf{1}_{\mathcal{O}_v^{\times}}(a_v)\mathbf{1}_{\mathcal{O}_v}(b_v). 
\end{align*}
Since $\tau_v$ lies in the maximal ideal of $\mathcal{O}_v$, the element $\sqrt{\tau_v}$ lies in the maximal ideal of $\mathcal{O}_{E_v}$. Thus $x_v=a_v+b_v\sqrt{\tau_v}\in \mathcal{O}_{E_v}^{\times}$ is equivalent to $a_v\in \mathcal{O}_v^{\times}$ and $b_v\in \mathcal{O}_v$. Hence
\begin{equation}\label{3.6}
\Phi_v((0,1)\iota(x_v))=q_v^{d_{E_v}/2}\eta_v(a_v)\mathbf{1}_{\mathcal{O}_{E_v}^{\times}}(x_v).  
\end{equation}
Moreover, $\chi_v^*$ has conductor exponent $1$, and by \eqref{e3.2},
\begin{equation}\label{3.5}
\chi_v^*(a_v+b_v\sqrt{\tau_v})=\chi_v^*(a_v)=\eta_v(a_v). 
\end{equation} 
Since $\mu_v$ is unramified, \eqref{3.6} and \eqref{3.5} imply
\begin{equation}\label{3.8}
Z_v(s,\Phi_v,\widetilde{\mu}_v\chi_v^*\chi_v)=q_v^{d_{E_v}/2}\int_{\mathcal{O}_{E_v}^{\times}}
\chi_v(x_v)d^{\times}x_v=\mathbf{1}_{\text{$\chi_v$ unramified}}.
\end{equation}

\item Suppose $v\in \Sigma_{F,\fin}$ with $v\mid \mathfrak{f}$. Then
\begin{align*}
\Phi_v((0,1)\iota(x_v))=\Phi_v(b_v,a_v)=\widetilde{\mu}_v(x_v)\mathbf{1}_{\mathcal{O}_v}(a_v)\mathbf{1}_{\mathcal{O}_v}(b_v)\mathbf{1}_{\mathcal{O}_v^{\times}}(\mathrm{Nr}_{E_v/F_v}(x_v)).
\end{align*}
Note that $x_v\in \mathcal{O}_{E_v}^{\times}$ is equivalent to $a_v, b_v\in \mathcal{O}_v$ and $\mathrm{Nr}_{E_v/F_v}(x_v)\in \mathcal{O}_v^{\times}$. Since $\chi_v^*$ is unramified, we obtain
\begin{equation}\label{3.9}
Z_v(s,\Phi_v,\widetilde{\mu}_v\chi_v^*\chi_v)=q_v^{d_{E_v}/2}\int_{\mathcal{O}_{E_v}^{\times}}
\chi_v(x_v)d^{\times}x_v=\mathbf{1}_{\text{$\chi_v$ unramified}}.
\end{equation}

\item Suppose $v\in \Sigma_{F,\fin}$ with $v\nmid \mathfrak{f}\mathfrak{D}_{E}$. Then 
\begin{align*}
\mathcal{O}_{E_v}=\mathcal{O}_v\oplus \mathfrak{a}_v^{-1}\theta_v,\quad \theta_v^2=\tau_v.
\end{align*} 
Therefore, 
\begin{align*}
\Phi_v((0,1)\iota(x_v))=\Phi_v(b_v,a_v)=q_v^{d_{E_v}/2}\mathbf{1}_{\mathcal{O}_v}(a_v)\mathbf{1}_{\mathfrak{a}_v^{-1}}(b_v)=q_v^{d_{E_v}/2}\mathbf{1}_{\mathcal{O}_{E_v}}(x_v). 
\end{align*}
It follows that
\begin{align*}
Z_v(s,\Phi_v,\widetilde{\mu}_v\chi_v^*\chi_v)=q_v^{d_{E_v}/2}\int_{\mathcal{O}_{E_v}-\{0\}}
\widetilde{\mu}_v\chi_v^*\chi_v(x_v)|\mathrm{Nr}_{E_v/F_v}(x_v)|_v^{1/2+s}d^{\times}x_v.
\end{align*}
Since $\widetilde{\mu}_v\chi_v^*$ is unramified, we get
\begin{equation}\label{3.10}
Z_v(s,\Phi_v,\widetilde{\mu}_v\chi_v^*\chi_v)=L_v(1/2+s,\widetilde{\mu}_v\chi_v^*\chi_v)\mathbf{1}_{\text{$\chi_v$ unramified}}. 
\end{equation}
\end{itemize}

Finally, for $v\in \Sigma_{F,\fin}$, if $\widetilde{\mu}_v\chi_v^*\chi_v$ is ramified, then
\begin{equation}\label{3.11}
L_v(1/2+s,\widetilde{\mu}_v\chi_v^*\chi_v)\equiv 1.
\end{equation}
Combining \eqref{e3.4}, \eqref{3.4}, \eqref{3.8}, \eqref{3.9}, \eqref{3.10}, and \eqref{3.11}, we obtain \eqref{e3.3} for $\Re(s)\gg 1$. The general case follows by meromorphic continuation.
\end{proof}

\section{Explicit Formulas of $J_{\mathrm{Const}}(h(\cdot,s))$ and $J_{\mathrm{Res}}(h(\cdot,s))$}\label{sec6}
Let $\mu$ and $\Phi$ be constructed as in \textsection\ref{sec3}. In this section, we explicitly compute the constant terms $h(I_2,s)$ and $h^{\Diamond}(I_2,s)$, as well as the residual terms $h(w,s)$ and $h^{\Diamond}(w,s)$, appearing on the right-hand side of Theorem \ref{thm3.1}.
 
\subsection{The Constant Term $J_{\mathrm{Const}}(h(\cdot,s))$}
\begin{prop}\label{prop6.1}
For $s\in \mathbb{C}$, we have 
\begin{align*}
J_{\mathrm{Const}}(h(\cdot,s))=\frac{\Gamma_{\mathbb{C}}(s+k+1)^{d}D_{E/F}^{1/2}D_F^{1/2}}{2^{d}}L^{(\mathfrak{f})}(1+2s,\eta).
\end{align*}	
Here $L^{(\mathfrak{f})}(1+2s,\eta)$ denotes the partial $L$-function with the Euler factors at primes dividing $\mathfrak{f}$ removed.
\end{prop}

Proposition \ref{prop6.1} follows immediately from Lemmas \ref{lem5.1} and \ref{lem5.2}.

\begin{lemma}\label{lem5.1}
For $s\in \mathbb{C}$, we have
\begin{equation}\label{5.1}
h(I_2,s)=\frac{\Gamma_{\mathbb{C}}(s+k+1)^{d}D_{E}^{1/2}}{2^{d}D_F^{1/2}}L^{(\mathfrak{f})}(1+2s,\eta). 
\end{equation}
\end{lemma}
\begin{proof}
It suffices to prove the formula for $\Re(s)\gg 1$. By \eqref{e2.2}, we have
\begin{align*}
h(I_2,s)=\prod_{v\in \Sigma_F}h_v(I_2,s),
\end{align*}
where 
\begin{align*}
h_v(I_2,s):=\int_{F_v^{\times}}
\Phi_v((0,t_v))\eta_v(t_v)|t_v|_v^{1+2s}d^{\times}t_v.
\end{align*}

We compute these local integrals using the definition of $\Phi_v(\cdot,\cdot)$ in \textsection\ref{sec3.1.3}.
\begin{itemize}
\item Suppose $v\in \Sigma_{F,\infty}$. Then 
\begin{align*}
h_v(I_2,s)=\int_{F_v^{\times}}
t_v^{2k+1}e^{-2\pi t_v^2}\sgn(t_v)|t_v|_v^{1+2s}d^{\times}t_v=\int_{\mathbb{R}^{\times}}
t^{2k+1}e^{-2\pi t^2}\sgn(t)|t|^{1+2s}d^{\times}t.
\end{align*}
Evaluating the integral gives
\begin{equation}\label{5.2}
h_v(I_2,s)=2^{-1}\Gamma_{\mathbb{C}}(s+k+1).	
\end{equation}

\item Suppose $v\in \Sigma_{F,\fin}$ with $v\mid \mathfrak{D}_{E}$. Since
\begin{align*}
\Phi_v(0,t_v)=\Phi_v(b_v,a_v)=q_v^{d_{E_v}/2}\eta_v(t_v)\mathbf{1}_{\mathcal{O}_v^{\times}}(t_v),
\end{align*}
we obtain
\begin{equation}\label{5.3}
h_v(I_2,s)=q_v^{d_{E_v}/2}\int_{\mathcal{O}_v^{\times}}
d^{\times}t_v=q_v^{(d_{E_v}-d_{F_v})/2}.
\end{equation}

\item Suppose $v\in \Sigma_{F,\fin}$ with $v\mid \mathfrak{f}$. Then $\Phi_v(0,t_v)=\mathbf{1}_{\mathcal{O}_v^{\times}}(t_v)$. Hence
\begin{equation}\label{5.4}
h_v(I_2,s)=\int_{\mathcal{O}_v^{\times}}
\eta_v(t_v)|t_v|_v^{1+2s}d^{\times}t_v=1.
\end{equation}

\item Suppose $v\in \Sigma_{F,\fin}$ with $v\nmid \mathfrak{f}\mathfrak{D}_{E}$. Since
\begin{align*}
\Phi_v(0,t_v)=q_v^{d_{E_v}/2}\mathbf{1}_{\mathcal{O}_v}(t_v), 
\end{align*}
we have
\begin{equation}\label{5.5}
h_v(I_2,s)=q_v^{d_{E_v}/2}\int_{\mathcal{O}_v-\{0\}}\eta_v(t_v)|t_v|_v^{1+2s}d^{\times}t_v=q_v^{(d_{E_v}-d_{F_v})/2}L_v(1+2s,\eta_v).
\end{equation}
\end{itemize}

Combining \eqref{5.2}, \eqref{5.3}, \eqref{5.4}, and \eqref{5.5} gives \eqref{5.1}. The general case follows by meromorphic continuation.
\end{proof}

\begin{lemma}\label{lem5.2}
For $s\in \mathbb{C}$, we have
\begin{equation}\label{5.6}
h^{\Diamond}(I_2,s)\equiv 0.  
\end{equation}
\end{lemma}
\begin{proof}
It suffices to prove the claim for $\Re(s)\gg 1$. By the definition of $h^{\Diamond}(\cdot,s)$ in \textsection\ref{sec2.1.2},
\begin{equation}\label{e5.7}
h^{\Diamond}(I_2,s)=\prod_{v\in \Sigma_F}h_v^{\Diamond}(I_2,s),
\end{equation}
where
\begin{align*}
h_v^{\Diamond}(I_2,s):=\int_{F_v}\int_{F_v^{\times}}
\Phi_v(t_v,c_v)\eta_v(t_v)|t_v|_v^{2s}d^{\times}t_vdc_v.
\end{align*}
The product is absolutely convergent in this region.

Choose a finite place $v\mid \mathfrak{D}_{E/F}$. By the definition of $\Phi_v(t_v,c_v)$, the corresponding local factor is
\begin{align*}
h_v^{\Diamond}(I_2,s)
=
q_v^{d_{E_v}/2}\int_{F_v}\int_{F_v^{\times}}
\eta_v(c_v)\mathbf{1}_{\mathcal{O}_v^{\times}}(c_v)\mathbf{1}_{\mathcal{O}_v}(t_v)
\eta_v(t_v)|t_v|_v^{2s}d^{\times}t_vdc_v.
\end{align*}
The $c_v$-integral vanishes, since $\eta_v$ is ramified. Hence $h_v^{\Diamond}(I_2,s)=0$, and therefore $h^{\Diamond}(I_2,s)=0$ for $\Re(s)\gg 1$. The identity for all $s$ follows by meromorphic continuation.
\end{proof}

\subsection{The Residual Term $J_{\mathrm{Res}}(h(\cdot,s))$}
\begin{prop}\label{prop6.2}
Let $\epsilon_{\mu,k}:=\mu_{\infty}(-1)(-1)^{kd}$. Then, for $s\in \mathbb{C}$, we have
\begin{align*}
J_{\mathrm{Res}}(h(\cdot,s))=
\frac{\epsilon_{\mu,k}\Gamma_{\mathbb{C}}(k+1-s)^{d}
D_{E/F}^{1/2-2s}D_F^{1/2-2s}}{2^{(1+2s)d}N_F(\mathfrak{f})^{2s}}L^{(\mathfrak{f})}(1-2s,\eta).
\end{align*}
\end{prop}

Proposition \ref{prop6.2} follows immediately from Lemmas \ref{lem6.2} and \ref{lem6.6}, together with the identity $\Gamma_{F_v}(1)=\Gamma_{\mathbb{R}}(1)=1$ for every $v\in \Sigma_{F,\infty}$, since $F$ is totally real.

\begin{lemma}\label{lem6.2}
For $s\in \mathbb{C}$, we have
\begin{equation}
\label{6.8}
h^{\Diamond}(w,s)=
\frac{\epsilon_{\mu,k}\Gamma_{\mathbb{C}}(k+1-s)^{d}
D_{E/F}^{1/2-2s}D_F^{1/2-2s}}{2^{(1+2s)d}
N_F(\mathfrak{f})^{2s}
}L^{(\mathfrak{f})}(1-2s,\eta).
\end{equation}
\end{lemma}
\begin{proof}
We first assume $\Re(s)\gg 1$. By the definition of $h^{\Diamond}(\cdot,s)$ in \textsection\ref{sec2.1.2},
\begin{align*}
h^{\Diamond}(w,s)=\prod_{v\in \Sigma_F}h_v^{\Diamond}(w,s),
\end{align*}
where
\begin{align*}
h_v^{\Diamond}(w,s):=\int_{F_v^{\times}}\int_{F_v}
\Phi_v(c_v,t_v)dc_v\eta_v(t_v)|t_v|_v^{2s}d^{\times}t_v.
\end{align*}
The product is absolutely convergent in this region. We compute the local factors from the definition of $\Phi_v(\cdot,\cdot)$ in \textsection\ref{sec3.1.3}.

Suppose first that $v\in \Sigma_{F,\infty}$. Put $A=|\tau_v|_v^{1/2}$. After the change of variables $c\mapsto Ac$, we get
\begin{align*}
h_v^{\Diamond}(w,s)=A^{-1}\int_{\mathbb{R}^{\times}}\int_{\mathbb{R}}
(t-ci)^{2k+1}e^{-2\pi(t^2+c^2)}dc\sgn(t)|t|^{2s}d^{\times}t.
\end{align*}
Passing to polar coordinates $t=r\cos\theta$, $c=r\sin\theta$, so that $t-ic=re^{-i\theta}$ and $dcdt=rdrd\theta$, gives
\begin{equation}\label{6.11}
h_v^{\Diamond}(w,s)
=
\frac{1}{A}
\int_0^{\infty}\frac{r^{2k+2s+1}}{e^{2\pi r^2}}dr
\int_0^{2\pi}
e^{-i(2k+1)\theta}\sgn(\cos\theta)|\cos\theta|^{2s-1}d\theta.
\end{equation}
The radial integral equals $2^{-1}(2\pi)^{-k-s-1}\Gamma(k+s+1)$. For the angular integral, parity and the oddness of $2k+1$ give
\begin{align*}
\int_0^{2\pi}
e^{-i(2k+1)\theta}\sgn(\cos\theta)|\cos\theta|^{2s-1}d\theta
=
4\int_0^{\frac{\pi}{2}}\cos((2k+1)\theta)\cos^{2s-1}\theta\,d\theta.
\end{align*}
Using
\begin{align*}
\int_0^{\frac{\pi}{2}}\cos(a\theta)\cos^{\nu-1}\theta\,d\theta
=
\frac{\pi\Gamma(\nu)}
{2^{\nu}\Gamma\left(\frac{\nu+1+a}{2}\right)
\Gamma\left(\frac{\nu+1-a}{2}\right)}
\end{align*}
with $\nu=2s$ and $a=2k+1$, we obtain
\begin{align*}
4\int_0^{\frac{\pi}{2}}\cos((2k+1)\theta)\cos^{2s-1}\theta\,d\theta
=
\frac{4\pi\Gamma(2s)}
{2^{2s}\Gamma(s+k+1)\Gamma(s-k)}.
\end{align*}
Substituting these two evaluations into \eqref{6.11} yields
\begin{equation}\label{6.12}
h_v^{\Diamond}(w,s)=
2^{-2s}(2\pi)^{-k-s}A^{-1}
\frac{\Gamma(2s)}{\Gamma(s-k)}.
\end{equation}

We now turn to the finite places. 
\begin{itemize}
\item Suppose $v\mid \mathfrak{D}_{E}$, then by the local definition of $\Phi_v$,
\begin{align*}
h_v^{\Diamond}(w,s)
&=
q_v^{d_{E_v}/2}\int_{F_v^{\times}}\int_{F_v}
\eta_v(t_v)\mathbf{1}_{\mathcal{O}_v^{\times}}(t_v)\mathbf{1}_{\mathcal{O}_v}(c_v)dc_v\eta_v(t_v)|t_v|_v^{2s}d^{\times}t_v \\
&=
q_v^{d_{E_v}/2-d_{F_v}}L_v(2s,\eta_v).
\end{align*}
Thus
\begin{equation}\label{6.13}
h_v^{\Diamond}(w,s)=q_v^{d_{E_v}/2-d_{F_v}}L_v(2s,\eta_v).
\end{equation}

\item Suppose $v\mid \mathfrak{f}$. Then
\begin{align*}
h_v^{\Diamond}(w,s)=\int_{\mathcal{O}_v-\{0\}}\int_{\mathcal{O}_v}
\mu_v(t_v^2-c_v^2\tau_v)\mathbf{1}_{\mathcal{O}_v^{\times}}(t_v^2-c_v^2\tau_v)dc_v
\eta_v(t_v)|t_v|_v^{2s}d^{\times}t_v.
\end{align*}
Since $\eta_v$ is unramified, writing $t_v=\varpi_v^n\alpha$ gives
\begin{equation}\label{6.14}
h_v^{\Diamond}(w,s)=\sum_{n\geq 0}\eta_v(\varpi_v)^nq_v^{-2ns}\int_{\mathcal{O}_v^{\times}}J_n(\alpha)d\alpha,
\end{equation}
where
\begin{align*}
J_n(\alpha):=\int_{\mathcal{O}_v}
\mu_v(\varpi_v^{2n}\alpha^2-c_v^2\tau_v)
\mathbf{1}_{\mathcal{O}_v^{\times}}(\varpi_v^{2n}\alpha^2-c_v^2\tau_v)dc_v.
\end{align*}
For $n=0$, the change $c_v\mapsto \alpha c_v$ and the standard finite-field identity
\begin{align*}
\sum_{\beta\in \mathbb{F}_v}\mu_v(1-\tau_v \beta^2)
=
-\mu_v(-\tau_v)
\end{align*}
give
\begin{equation}\label{6.15}
J_0(\alpha)=-q_v^{-1}\mu_v(-\tau_v).
\end{equation}
For $n\geq 1$, the unit condition is equivalent to $c_v\in \mathcal{O}_v^{\times}$, and
\begin{align*}
\varpi_v^{2n}\alpha^2-c_v^2\tau_v
=
-c_v^2\tau_v
\left(1-\frac{\varpi_v^{2n}\alpha^2}{c_v^2\tau_v}\right).
\end{align*}
The last factor lies in $1+\mathfrak{p}_v^2$ and is killed by $\mu_v$, while $\mu_v(c_v^2)=1$. Hence
\begin{equation}\label{6.16}
J_n(\alpha)=(1-q_v^{-1})\mu_v(-\tau_v).
\end{equation}
Substituting \eqref{6.15} and \eqref{6.16} into \eqref{6.14}, we obtain
\begin{align*}
h_v^{\Diamond}(w,s)=\mu_v(-\tau_v)
\left(\eta_v(\varpi_v)q_v^{-2s}-q_v^{-1}\right)
L_v(2s,\eta_v).
\end{align*}
Since $\mathfrak{f}$ is prime to $2\mathfrak{D}_E$, the place $v$ is either split or inert unramified. If $v$ is split, then $\eta_v=1$ and $\tau_v$ is a square in $\mathcal{O}_v^{\times}$, so
\begin{align*}
h_v^{\Diamond}(w,s)=\mu_v(-1)
q_v^{-2s}\left(1-q_v^{-1+2s}\right)
L_v(2s,\eta_v).
\end{align*}
If $v$ is inert unramified, then $\eta_v(\varpi_v)=-1$, $\tau_v$ is a nonsquare unit, and $\mu_v(\tau_v)=-1$. Thus
\begin{align*}
h_v^{\Diamond}(w,s)=\mu_v(-1)
q_v^{-2s}\left(1+q_v^{-1+2s}\right)
L_v(2s,\eta_v).
\end{align*}
The two cases combine into
\begin{equation}\label{6.20}
h_v^{\Diamond}(w,s)=\mu_v(-1)
q_v^{-2s}L_v(1-2s,\eta_v)^{-1}
L_v(2s,\eta_v).
\end{equation}

\item Finally, if $v\nmid \mathfrak{f}\mathfrak{D}_{E}$, then directly from the unramified definition of $\Phi_v(\cdot,\cdot)$,
\begin{align*}
h_v^{\Diamond}(w,s)=q_v^{d_{E_v}/2}\int_{F_v^{\times}}\int_{F_v}
\mathbf{1}_{\mathcal{O}_v}(t_v)\mathbf{1}_{\mathfrak{a}_v^{-1}}(c_v)dc_v\eta_v(t_v)|t_v|_v^{2s}d^{\times}t_v.
\end{align*}
Therefore, 
\begin{equation}\label{6.25}
h_v^{\Diamond}(w,s)=q_v^{d_{E_v}/2-d_{F_v}}q_v^{e_v(\mathfrak{a})}L_v(2s,\eta_v).
\end{equation}
\end{itemize}

Multiplying \eqref{6.12}, \eqref{6.13}, \eqref{6.20}, and \eqref{6.25}, and using $\prod_{v\mid\mathfrak{f}}\mu_v(-1)=\mu_{\infty}(-1)$, we obtain, for $\Re(s)\gg 1$,
\begin{multline}\label{6.26}
h^{\Diamond}(w,s)=\mu_{\infty}(-1)N_F(\mathfrak{f})^{-2s}\Big[ 2^{-2s}(2\pi)^{-k-s}
\frac{\Gamma(2s)}{\Gamma(s-k)} \Big]^{d}\\
|\tau_{\infty}|_{\infty}^{-1/2}N_F(\mathfrak{a})D_{E}^{1/2}D_F^{-1}L(2s,\eta)\prod_{v\mid\mathfrak{f}}L_v(1-2s,\eta_v)^{-1}.
\end{multline}
By meromorphic continuation, \eqref{6.26} holds for all $s\in \mathbb{C}$.

By the definition \eqref{4.3}, we have
\begin{equation}\label{6.22}
|\tau_{\infty}|_{\infty}=N_F((\tau))=N_F(\mathfrak{D}_{E/F})N_F(\mathfrak{a})^2=D_{E/F}N_F(\mathfrak{a})^2.
\end{equation}
Substituting \eqref{6.22} into \eqref{6.26}, the factors involving $D_{E/F}$ and $N_F(\mathfrak{a})$ cancel, giving
\begin{multline}\label{6.26.}
h^{\Diamond}(w,s)=\mu_{\infty}(-1)N_F(\mathfrak{f})^{-2s}\Big[ 2^{-2s}(2\pi)^{-k-s}
\frac{\Gamma(2s)}{\Gamma(s-k)} \Big]^{d}\\
L(2s,\eta)\prod_{v\mid\mathfrak{f}}L_v(1-2s,\eta_v)^{-1}.
\end{multline}
The functional equation gives
\begin{equation}\label{6.27}
L(2s,\eta)=
D_F^{\frac{1}{2}-2s}D_{E/F}^{\frac{1}{2}-2s}L(1-2s,\eta)
\bigg[
\frac{\Gamma_{\mathbb{R}}(2-2s)}
{\Gamma_{\mathbb{R}}(2s+1)}
\bigg]^{d}.
\end{equation}
Moreover,
\begin{equation}\label{6.28}
2^{-2s}(2\pi)^{-k-s}
\frac{\Gamma(2s)}{\Gamma(s-k)}
\cdot 
\frac{\Gamma_{\mathbb{R}}(2-2s)}
{\Gamma_{\mathbb{R}}(2s+1)}
=(-1)^k2^{-2s-1}\Gamma_{\mathbb{C}}(k+1-s).
\end{equation}
Combining \eqref{6.26.}, \eqref{6.27}, and \eqref{6.28} gives \eqref{6.8}. This proves the lemma.
\end{proof}

\begin{lemma}\label{lem6.6}
For $s\in \mathbb{C}$, we have
\begin{equation}\label{e6.24}
h(w,s)\equiv 0.  
\end{equation}
\end{lemma}

\begin{proof}
Again it suffices to prove the identity for $\Re(s)\gg 1$. By definition,
\begin{equation}\label{e6.25}
h(w,s)=\prod_{v\in \Sigma_F}h_v(w,s),
\end{equation}
where
\begin{align*}
h_v(w,s):=\int_{F_v^{\times}}
\Phi_v(t_v,0)\eta_v(t_v)|t_v|_v^{1+2s}d^{\times}t_v.
\end{align*}
The product is absolutely convergent in this region.

For any finite place $v\mid \mathfrak{D}_{E/F}$, the local test function satisfies $\Phi_v(t_v,0)\equiv 0$ by construction. Thus $h_v(w,s)=0$, and consequently $h(w,s)=0$ for $\Re(s)\gg 1$. The general case follows by meromorphic continuation.
\end{proof}

\section{Estimates for $J_{\mathrm{Dual}}^{\heartsuit}(h(\cdot,s))$}\label{sec7}
In this section we estimate the dual side $J_{\mathrm{Dual}}^{\heartsuit}(h(\cdot,s))$ appearing on the right-hand side of Theorem \ref{thm3.1}. Recall that
\begin{equation}\label{7.1}
J_{\mathrm{Dual}}^{\heartsuit}(h(\cdot,s)):=\frac{1}{2\pi\underset{\lambda=1}{\Res}\ \zeta_F(\lambda)}\sum_{\xi\in \widehat{C}_F}\int_{\mathbb{R}}\widetilde{\Psi}(W(\cdot,s),\xi|\cdot|^{it})dt.
\end{equation}

For an integral ideal $\mathfrak{m}\subseteq\mathcal{O}_F$, let $\widehat{C}_F(\mathfrak{m})$ denote the subset of $\widehat{C}_F$ consisting of characters of arithmetic conductor $\mathfrak{m}$. 

For $t\in \mathbb{R}$ and $\xi=\otimes_{v\in \Sigma_{F,v}}'\xi_v\in \widehat{C}_F$, at $v\in \Sigma_{F,\infty}$, let $\boldsymbol{\nu}_{\xi_v,t}$ be the spectral parameter of $\xi_v|\cdot|_v^{it}$; thus $\boldsymbol{\nu}_{\xi_v,t}=t+\nu_v$ if $\xi_v=\sgn^{j}|\cdot|^{i\nu_v}$. For $R>0$, define the spectral box of size $R$ by
\begin{align*}
\mathbf{B}(\boldsymbol{\nu}_{\xi_{\infty},t},R):=\Big\{|\boldsymbol{\nu}_{\xi_v,t}|\leq R:\ v\in \Sigma_{F,\infty}\Big\}.
\end{align*}

Let $0<\delta<1/2$, $|\Re(s)|\leq \delta$, and $0<\varepsilon<10^{-3}$. Define
\begin{multline}\label{e7.2}
\mathcal{J}_{\mathrm{Split}}(s,\eta;\mathfrak{f}):=\sum_{\mathfrak{m}\supseteq \mathfrak{f}}N_F(\mathfrak{m})^{-1/2+\delta}\sum_{\xi\in \widehat{C}_F(\mathfrak{m})}\\
\int_{\mathbf{B}(\boldsymbol{\nu}_{\xi_{\infty},t},QN_F(\mathfrak{f}\mathfrak{D}_{E})^{\varepsilon})}\big|L(1/2+it,\xi)L(1/2+it,\eta\xi)\big|dt.
\end{multline}

\begin{prop}\label{prop7.1}
Let $0<\delta<1/2$ and $|\Re(s)|\leq \delta$. Let $Q=1+2k+|s|$ and $0<\varepsilon<10^{-3}$. Then
\begin{multline}\label{eq7.3}
J_{\mathrm{Dual}}^{\heartsuit}(h(\cdot,s))\ll_{F,\delta,\varepsilon} \frac{D_{E/F}^{\frac{1}{4}+2\delta}|\Gamma_{\mathbb C}(k+1+s)|^{d}}{\big(Q^{\frac{1}{4}-2\delta}M^{-\varepsilon\delta}\big)^{d}}\cdot\mathcal{J}_{\mathrm{Split}}(s,\eta;\mathfrak{f})\\
+|\Gamma_{\mathbb C}(k+1+s)|^{d}Q^{-100}N_F(\mathfrak{f}\mathfrak{D}_{E})^{-100}.
\end{multline}
In particular, we have
\begin{equation}\label{eq7.4}
J_{\mathrm{Dual}}^{\heartsuit}(h(\cdot,s))\ll_{F,\delta,\varepsilon}  |\Gamma_{\mathbb C}(k+1+s)|^{d}D_{E/F}^{\frac{7}{16}+2\delta+\varepsilon}
Q^{\frac{15d}{16}+2\delta d+\varepsilon}N_F(\mathfrak{f})^{\frac{11}{16}+\delta+\varepsilon}.
\end{equation}
\end{prop}
\begin{remark}
The implied constants in \eqref{eq7.3} and \eqref{eq7.3} are independent of
$\mathfrak{a}$, and hence independent of the choice of $\tau$ used in the
parametrization $E=F(\sqrt{\tau})$.
\end{remark}

\subsection{Structure of Split Period Integrals}

Let $\pi=|\cdot|^s\boxplus \eta|\cdot|^{-s}$, and write $\pi=\otimes_{v\in \Sigma_F}'\pi_v$ for its local factorization. By \eqref{e2.6}, we have 
\begin{equation}\label{7.2}
\widetilde{\Psi}(W(\cdot,s),\mu\xi|\cdot|^{it})=\Lambda(1/2+it,\pi\otimes\xi)\prod_{v\in \Sigma_F}\frac{\mathcal{P}_v(s,it;\xi_v)}{L_v(1/2+it,\pi_v\otimes\xi_v)},
\end{equation}
where
\begin{align*}
\mathcal{P}_v(s,it;\xi_v):=\int_{F_v^{\times}}W_v^*\left(\begin{pmatrix}
y_v\\
& 1
\end{pmatrix}
,s\right)\xi_v(y_v)|y_v|_v^{it}d^{\times}y_v. 
\end{align*}
Here
\begin{align*}
W_v^*(g_v,s):=W_v(g_v,s)\mu_v(\det g_v),
\quad g_v\in G(F_v).
\end{align*}

By \eqref{2.2}, this twisted Whittaker function $W_v^*(g_v,s)$ is given by
\begin{equation}\label{e7.3}
|\det g_v|_v^{1/2+s}
\int_{F_v^{\times}}\int_{F_v}
\Phi_v((a_v,c_va_v)g_v)\overline{\psi}_v(c_v)dc_v
\eta_v(a_v)|a_v|_v^{1+2s}d^{\times}a_v.
\end{equation}

\subsection{Archimedean Integrals}\label{sec7.1.1}
Suppose $v\in \Sigma_{F,\infty}$. Then, for $y\in F_v\simeq\mathbb{R}$,
\begin{align*}
W_v^*\left(\begin{pmatrix}
y\\
& 1
\end{pmatrix}
,s\right)=|y|^{1/2+s}
\int_{\mathbb{R}^{\times}}\int_{\mathbb{R}}
\Phi_v((ya,ca))e^{-2\pi i c}dc
\sgn(a)|a|^{1+2s}d^{\times}a.
\end{align*}
Let $A=|\tau_v|_v^{1/2}$ and $Q=1+2k+|s|$. By the definition of $\Phi_v(\cdot,\cdot)$, we obtain
\begin{multline}\label{7.3}
\mathcal{P}_v(s,it;\xi_v)=\int_{\mathbb{R}^{\times}}
\int_{\mathbb{R}^{\times}}\int_{\mathbb{R}}
(c-yiA)^{2k+1}e^{-2\pi a^2(c^2+y^2A^2)}e^{-2\pi i c}dc\\
|a|^{2k+2+2s}d^{\times}a\xi_v(y)|y|^{1/2+s+it}d^{\times}y.
\end{multline}
Write $\xi_v(y)=\sgn(y)^{j}|y|^{i\nu}$, with $j\in\{0,1\}$, and put $T:=t+\nu$. Define
\begin{align*}
\mathcal{J}_{k}^{-}(s,T)
&:=
\int_0^{\frac{\pi}{2}}
\sin^{-1/2+s-iT}\theta
\cos^{-1/2+s+iT}\theta
\sin((2k+1)\theta)d\theta,\\
\mathcal{J}_{k}^{+}(s,T)
&:=
\int_0^{\frac{\pi}{2}}
\sin^{-1/2+s-iT}\theta
\cos^{-1/2+s+iT}\theta
\cos((2k+1)\theta)d\theta.
\end{align*}

\begin{lemma}\label{lem7.2}
Suppose $v\in \Sigma_{F,\infty}$. Then
\begin{multline}\label{e7.4}
\frac{\mathcal{P}_v(s,it;\xi_v)}{\Gamma_{\mathbb C}(k+1+s)}=
2(-1)^{k+1}i
A^{-1/2-s-iT}
(2\pi)^{-1/2+s-iT}
\Gamma\left(1/2-s+iT\right) \\
\times
\begin{cases}
\displaystyle
\sin\left(\frac{\pi}{2}\left(1/2-s+iT\right)\right)
\mathcal{J}_{k}^{-}(s,T),
& j\equiv 0\pmod 2,\\[1em]
\displaystyle
\cos\left(\frac{\pi}{2}\left(1/2-s+iT\right)\right)
\mathcal{J}_{k}^{+}(s,T),
& j\equiv 1\pmod 2.
\end{cases}
\end{multline}
\end{lemma}

\begin{proof}
The integral \eqref{7.3} is
\begin{multline*}
\mathcal{P}_v(s,it;\xi_v)=\int_{\mathbb{R}^{\times}}
\int_{\mathbb{R}^{\times}}\int_{\mathbb{R}}
(c-yiA)^{2k+1}e^{-2\pi a^2(c^2+y^2A^2)}e^{-2\pi i c}dc\\
|a|^{2k+2+2s}d^{\times}a\sgn(y)^{j}|y|^{1/2+s+iT}d^{\times}y. 
\end{multline*}
Since $d^{\times}a=da/|a|$, the $a$-integral equals
\begin{align*}
\int_{\mathbb{R}^{\times}}
e^{-2\pi a^2(c^2+y^2A^2)}
|a|^{2k+2+2s}d^{\times}a
=
(2\pi)^{-k-1-s}
\Gamma(k+1+s)
(c^2+y^2A^2)^{-k-1-s}.
\end{align*}
Using $\Gamma_{\mathbb C}(k+1+s)=2(2\pi)^{-k-1-s}\Gamma(k+1+s)$, we obtain
\begin{equation}\label{7.4}
\frac{\mathcal{P}_v(s,it;\xi_v)}{\Gamma_{\mathbb C}(k+1+s)}
=\frac{1}{2}
\int_{\mathbb R^{\times}}\int_{\mathbb R}
\frac{(c-yiA)^{2k+1}e^{-2\pi ic}}{(c^2+y^2A^2)^{k+1+s}}dc
\sgn(y)^j |y|^{1/2+s+iT}d^{\times}y.
\end{equation}

We use polar coordinates in the $(c,Ay)$-plane:
\begin{align*}
c=r\cos\theta,\quad Ay=r\sin\theta,
\quad r>0,\quad 0\leq \theta<2\pi.
\end{align*}
Then $c-yiA=re^{-i\theta}$, $c^2+y^2A^2=r^2$, and $dcdy=A^{-1}rdrd\theta$. Since $y=A^{-1}r\sin\theta$, we have
\begin{align*}
|y|^{1/2+s+iT}d^{\times}y\,dc
=
A^{-1/2-s-iT}
r^{1/2+s+iT}
|\sin\theta|^{-1/2+s+iT}dr\,d\theta.
\end{align*}
Also
\begin{align*}
\frac{(c-yiA)^{2k+1}}{(c^2+y^2A^2)^{k+1+s}}
=
r^{-1-2s}e^{-i(2k+1)\theta}.
\end{align*}
Substitution in \eqref{7.4} gives
\begin{multline*}
\frac{\mathcal{P}_v(s,it;\xi_v)}{\Gamma_{\mathbb C}(k+1+s)}
=
1/2 A^{-1/2-s-iT}
\int_0^\infty r^{-1/2-s+iT}
\int_0^{2\pi}
e^{-i(2k+1)\theta}e^{-2\pi ir\cos\theta}\\
\times \sgn(\sin\theta)^j
|\sin\theta|^{-1/2+s+iT}d\theta\,dr.
\end{multline*}
Put $\beta:=1/2-s+iT$. The $r$-integral is
\begin{align*}
\int_0^\infty r^{\beta-1}e^{-2\pi ir\cos\theta}dr
=
\Gamma(\beta)(2\pi|\cos\theta|)^{-\beta}
e^{-\frac{\pi i\beta}{2}\sgn(\cos\theta)},
\end{align*}
initially in a convergent range and then by meromorphic continuation. Hence
\begin{multline}\label{7.5}
\frac{\mathcal{P}_v(s,it;\xi_v)}{\Gamma_{\mathbb C}(k+1+s)}
=
1/2 A^{-1/2-s-iT}
(2\pi)^{-1/2+s-iT}
\Gamma(1/2-s+iT) \\
\times
\int_0^{2\pi}
e^{-i(2k+1)\theta}
\sgn(\sin\theta)^j
|\sin\theta|^{-1/2+s+iT}
|\cos\theta|^{-1/2+s-iT}
e^{-\frac{\pi i(1/2-s+iT)}{2}\sgn(\cos\theta)}
d\theta.
\end{multline}

Let $n=2k+1$ and, for $0<\theta<\pi/2$, set
\begin{align*}
F(\theta):=
\sin^{-1/2+s+iT}\theta
\cos^{-1/2+s-iT}\theta.
\end{align*}
Splitting the angular integral in \eqref{7.5} into four quadrants, and using that $n$ is odd, gives
\begin{align*}
&\int_0^{2\pi}
e^{-in\theta}
\sgn(\sin\theta)^j
|\sin\theta|^{-1/2+s+iT}
|\cos\theta|^{-1/2+s-iT}
e^{-\frac{\pi i\beta}{2}\sgn(\cos\theta)}
d\theta \\
&\quad =
\int_0^{\frac{\pi}{2}}
F(\theta)
\Bigg[
e^{-in\theta}
\left(
e^{-\frac{\pi i\beta}{2}}
-(-1)^j e^{\frac{\pi i\beta}{2}}
\right)
+
e^{in\theta}
\left(
(-1)^j e^{-\frac{\pi i\beta}{2}}
-e^{\frac{\pi i\beta}{2}}
\right)
\Bigg]d\theta.
\end{align*}
If $j\equiv 0\pmod 2$, the expression in brackets is $-4i\sin(\pi\beta/2)\cos(n\theta)$, whereas if $j\equiv 1\pmod 2$, it is $-4i\cos(\pi\beta/2)\sin(n\theta)$. Therefore
\begin{multline}\label{7.6}
\frac{\mathcal{P}_v(s,it;\xi_v)}{\Gamma_{\mathbb C}(k+1+s)}
=
-2i
A^{-1/2-s-iT}
(2\pi)^{-1/2+s-iT}
\Gamma\left(1/2-s+iT\right) \\
\times
\begin{cases}
\displaystyle
\sin\left(\frac{\pi}{2}\left(1/2-s+iT\right)\right)
\int_0^{\frac{\pi}{2}}F(\theta)\cos((2k+1)\theta)d\theta,
& j\equiv 0\pmod 2,\\[1em]
\displaystyle
\cos\left(\frac{\pi}{2}\left(1/2-s+iT\right)\right)
\int_0^{\frac{\pi}{2}}F(\theta)\sin((2k+1)\theta)d\theta,
& j\equiv 1\pmod 2.
\end{cases}
\end{multline}

We now rewrite the angular integrals in the stated form. In the case $j\equiv 0\pmod 2$, the change of variables $\theta\mapsto \pi/2-\theta$ gives
\begin{align*}
\int_0^{\frac{\pi}{2}}F(\theta)\cos((2k+1)\theta)d\theta
=
(-1)^k
\mathcal{J}_{k}^{-}(s,T).
\end{align*}
In the case $j\equiv 1\pmod 2$, the same change of variables gives
\begin{align*}
\int_0^{\frac{\pi}{2}}F(\theta)\sin((2k+1)\theta)d\theta
=
(-1)^k
\mathcal{J}_{k}^{+}(s,T).
\end{align*}
Substituting these identities into \eqref{7.6} gives \eqref{e7.4}.
\end{proof}

\begin{prop}\label{prop7.3}
Let $0<\delta<1/2$ and $|\Re(s)|\leq \delta$. Let $m\geq 0$ be an integer. Then
\begin{equation}\label{7.8}
\frac{\mathcal{P}_v(s,it;\xi_v)}{\Gamma_{\mathbb C}(k+1+s)}
\ll_{m,\delta}
\frac{\left(
Q^{-1/4+\delta}\mathbf{1}_{|T|\leq 2Q}
+
(1+|T|)^{-m}\mathbf{1}_{|T|>2Q}
\right)
}{A^{1/2+\Re(s)}
(1+|T-\Im(s)|)^{\Re(s)}
}.
\end{equation}
\end{prop}

\begin{proof}
We use the notation of Lemma \ref{lem7.2}. By Stirling's formula and the elementary bounds for sine and cosine in vertical strips, we have
\begin{equation}\label{7.10}
\left|
\Gamma\left(1/2-s+iT\right)F_1(s,T)
\right|
\ll_{\delta}
(1+|T-\Im(s)|)^{-\Re(s)},
\end{equation}
where
\begin{align*}
F_1(s,T):=
\left|\sin\left(\frac{\pi}{2}\left(1/2-s+iT\right)\right)\right|
+
\left|\cos\left(\frac{\pi}{2}\left(1/2-s+iT\right)\right)\right|.
\end{align*}
Combining \eqref{7.10} with Lemma \ref{lem7.2}, and applying the bounds for $\mathcal{J}_{k}^{\pm}(s,T)$ from Lemmas \ref{lemA.1}, \ref{lemA.2}, and \ref{lemA.3} in Appendix \textsection\ref{App}, we obtain \eqref{7.8}. Here, as usual, the factors $\sin(n\theta)$ and $\cos(n\theta)$ are written as linear combinations of $e^{\pm in\theta}$.
\end{proof}

\subsection{Non-Archimedean Integrals}

Let $v\in \Sigma_{F,\fin}$ and $\pi_v:=|\cdot|_v^s\boxplus \eta_v|\cdot|_v^{-s}$. By \eqref{e7.3}, we have
\begin{multline}\label{6.3}
\mathcal{P}_v(s,it;\xi_v)=\int_{F_v^{\times}}
\int_{F_v^{\times}}\int_{F_v}
\Phi_v(y_v,c_v)\overline{\psi}_v(c_va_v)dc_v\\
\xi_v\eta_v(a_v)|a_v|_v^{1/2-s+it}d^{\times}a_v\xi_v(y_v)|y_v|_v^{1/2+s+it}d^{\times}y_v.
\end{multline}

Let $r_{\xi_v}$ be the conductor exponent of $\xi_v$. 
\subsubsection{The case $v\nmid \mathfrak{f}\mathfrak{D}_{E}$}

Substituting the definition of $\Phi_v(\cdot,\cdot)$ into \eqref{6.3} gives
\begin{multline}\label{6.4}
\mathcal{P}_v(s,it;\xi_v)=q_v^{d_{E_v}/2}
\int_{\mathfrak{a}_v^{-1}-\{0\}}\int_{F_v^{\times}}\int_{\mathcal{O}_v}
\overline{\psi}_v(c_va_v)dc_v\xi_v\eta_v(a_v)\\
|a_v|_v^{1/2-s+it}d^{\times}a_v\xi_v(y_v)|y_v|_v^{1/2+s+it}d^{\times}y_v.
\end{multline}
Since $\eta_v$ is unramified, $\mathcal{P}_v(s,it;\xi_v)$ vanishes unless $\xi_v$ is unramified. More precisely, 
\begin{equation}\label{6.5}
\frac{\mathcal{P}_v(s,it;\xi_v)}{L_v(1/2+it,\pi_v\otimes\xi_v)}=\overline{\xi}_v(\varpi_v^{e_v(\mathfrak{a})})q_v^{(1/2+s+it)e_v(\mathfrak{a})}q_v^{\frac{d_{E_v}-d_{F_v}}{2}}\mathbf{1}_{r_{\xi_v}=0}.
\end{equation}

\subsubsection{The case $v\mid \mathfrak{D}_{E}$}

In parallel with \eqref{6.4}, we have
\begin{multline}\label{6.6}
\mathcal{P}_v(s,it;\xi_v)=q_v^{d_{E_v}/2}\int_{\mathcal{O}_v-\{0\}}
\int_{F_v^{\times}}\int_{\mathcal{O}_v^{\times}}\overline{\psi}_v(c_va_v)\eta_v(c_v)dc_v\\
\xi_v\eta_v(a_v)|a_v|_v^{1/2-s+it}d^{\times}a_v\xi_v(y_v)|y_v|_v^{1/2+s+it}d^{\times}y_v.
\end{multline}
Since $\eta_v$ has conductor exponent $1$, the $c_v$-integral vanishes unless $e_v(a_v)=1$. Hence, by \eqref{6.6} and a change of variables,
\begin{multline}\label{7.21}
\mathcal{P}_v(s,it;\xi_v)=\overline{\xi}_v\eta_v(\varpi_v)q_v^{d_{E_v}/2+1/2-s+it}
\int_{\mathcal{O}_v^{\times}}\overline{\psi}_v(\varpi_v^{-1}c_v)\eta_v(c_v)dc_v\\
\int_{\mathcal{O}_v^{\times}}\xi_v(a_v)d^{\times}a_v\int_{\mathcal{O}_v-\{0\}}\xi_v(y_v)|y_v|_v^{1/2+s+it}d^{\times}y_v.
\end{multline}
By definition,
\begin{equation}\label{7.22}
\int_{\mathcal{O}_v^{\times}}\overline{\psi}_v(\varpi_v^{-1}c_v)\eta_v(c_v)dc_v=q_v^{-1/2}\varepsilon(1/2,\eta_v,\psi_v).
\end{equation}
Here $\varepsilon(1/2,\eta_v,\psi_v)$ denotes the local root number. Therefore, \eqref{7.21} and \eqref{7.22} imply
\begin{equation}\label{fc7.21}
\frac{\mathcal{P}_v(s,it;\xi_v)}{L_v(1/2+it,\pi_v\otimes\xi_v)}=\overline{\xi}_v\eta_v(\varpi_v)q_v^{d_{E_v}/2-s+it}\varepsilon(1/2,\eta_v,\psi_v)\mathbf{1}_{r_{\xi_v}=0}.
\end{equation}
Here we used the fact that $L_v(1/2+it,\pi_v\otimes\xi_v)=L_v(1/2+it,\xi_v)$.

\subsubsection{Suppose $v\mid \mathfrak{f}$}
By the definition of $\mathfrak{f}$, we have $v\nmid 2$ and $\mu_v$ is quadratic. Since $v\nmid 2$, the group $1+\mathfrak p_v$ is pro-$p$ with $p\neq 2$. Hence every quadratic character of $F_v^{\times}$ is trivial on $1+\mathfrak p_v$. As $\mu_v$ is ramified, its conductor exponent is therefore exactly $1$, i.e., $e_v(\mathfrak{f})=1$. 
 
\begin{prop}\label{prop7.5}
Let $0<\delta<1/2$ and $|\Re(s)|\leq \delta$. Then 
\begin{equation}\label{f7.21}
\frac{\mathcal{P}_v(s,it;\xi_v)}{L_v(1/2+it,\pi_v\otimes\xi_v)}\ll_{\delta} q_v^{-\Re(s)}\mathbf{1}_{r_{\xi_v}=0}+q_v^{-1/2-\Re(s)}\mathbf{1}_{r_{\xi_v}=1}.	
\end{equation}	
\end{prop}
\begin{proof}
Substituting the definition of $\Phi_v(\cdot,\cdot)$ into \eqref{6.3} gives 
\begin{equation}\label{7.24}
\mathcal{P}_v(s,it;\xi_v)=\int_{\mathcal{O}_v-\{0\}}H(y_v)\xi_v(y_v)|y_v|_v^{1/2+s+it}d^{\times}y_v,
\end{equation}
where  
\begin{align*}
H(y_v):=\int_{F_v^{\times}}\int_{\mathcal{O}_v}
\mu_v(c_v^2-y_v^2\tau_v)\mathbf{1}_{\mathcal{O}_v^{\times}}(c_v^2-y_v^2\tau_v)
\overline{\psi}_v(c_va_v)dc_v
\xi_v\eta_v(a_v)|a_v|_v^{\frac{1}{2}-s+it}d^{\times}a_v .
\end{align*}

For $u_v\in 1+\mathfrak{p}_v$, $y_v\in \mathcal{O}_v$, and $\tau_v\in \mathcal{O}_v^{\times}$, we have  
\begin{equation}\label{e7.22}
\mu_v(c_v^2-u_v^2y_v^2\tau_v)=\mu_v(c_v^2-y_v^2\tau_v),
\end{equation}
and 
\begin{equation}\label{7.23}
c_v^2-y_v^2\tau_v\in\mathcal{O}_v^{\times}
\Longleftrightarrow
c_v^2-u_v^2y_v^2\tau_v\in\mathcal{O}_v^{\times}.
\end{equation}

It follows from \eqref{e7.22} and \eqref{7.23} that  
\begin{equation}\label{e7.24}
H(u_vy_v)=H(y_v),\quad u_v\in 1+\mathfrak{p}_v.
\end{equation}

Combining \eqref{e7.24} with \eqref{7.24} and using the orthogonality of characters, we obtain
\begin{equation}\label{7.25}
\mathcal{P}_v(s,it;\xi_v)\equiv 0\quad \text{unless}\quad  r_{\xi_v}\leq 1=e_v(\mathfrak{f}). 
\end{equation}

By the properties of Ramanujan sums and Gauss sums, the $a_v$-integral is supported in a compact subset of $F_v-\{0\}$ and hence converges for $\Re(s)<1/2$. We may therefore interchange the integrals in \eqref{7.24} and make the change of variables $a_v\mapsto c_v^{-1}a_v$, obtaining  
\begin{multline}\label{7.26}
\mathcal{P}_v(s,it;\xi_v)=G(\xi_v)\int_{\mathcal{O}_v-\{0\}}
\int_{\mathcal{O}_v}
\mu_v(c_v^2-y_v^2\tau_v)\mathbf{1}_{\mathcal{O}_v^{\times}}(c_v^2-y_v^2\tau_v)\\
\overline{\xi}_v\eta_v(c_v)|c_v|_v^{-1/2+s-it}dc_v\xi_v(y_v)|y_v|_v^{1/2+s+it}d^{\times}y_v,
\end{multline}
where 
\begin{align*}
G(\xi_v):=\int_{F_v^{\times}}\xi_v\eta_v(a_v)\overline{\psi}_v(a_v)|a_v|_v^{1/2-s+it}d^{\times}a_v.
\end{align*}

Again by the properties of Gauss sums and Ramanujan sums, we have 
\begin{equation}\label{eq7.28}
G(\xi_v)\mathbf{1}_{r_{\xi_v}\leq 1}\ll_{\delta} q_v^{1/2-\Re(s)}\mathbf{1}_{r_{\xi_v}=0}+q_v^{-\Re(s)}\mathbf{1}_{r_{\xi_v}=1}.
\end{equation}

We now decompose the double integral in \eqref{7.26} into the following regions.   
\begin{itemize}
\item If $y_v\in \mathcal{O}_v-\{0\}$ and $c_v\in \mathfrak{p}_v$, then 
\begin{align*}
\mu_v(c_v^2-y_v^2\tau_v)\mathbf{1}_{\mathcal{O}_v^{\times}}(c_v^2-y_v^2\tau_v)=\mu_v(y_v^2\tau_v)\mathbf{1}_{\mathcal{O}_v^{\times}}(y_v^2\tau_v)=\mathbf{1}_{\mathcal{O}_v^{\times}}(y_v).
\end{align*}

Thus the contribution of this region is
\begin{multline}\label{7.27}
G(\xi_v)\int_{\mathcal{O}_v^{\times}}
\int_{\mathfrak{p}_v}
\overline{\xi}_v\eta_v(c_v)|c_v|_v^{-1/2+s-it}dc_v\xi_v(y_v)d^{\times}y_v\\
=\overline{\xi}_v\eta_v(\varpi_v)G(\xi_v)
(1-q_v^{-1})
q_v^{-1/2-s+it}L_v(1/2+s-it,\overline{\xi}_v\eta_v)\mathbf{1}_{r_{\xi_v}=0}. 
\end{multline}  
	
\item If $y_v\in \mathfrak{p}_v-\{0\}$ and $c_v\in \mathcal{O}_v^{\times}$, then 
\begin{align*}
\mu_v(c_v^2-y_v^2\tau_v)\mathbf{1}_{\mathcal{O}_v^{\times}}(c_v^2-y_v^2\tau_v)=\mu_v(c_v^2)\mathbf{1}_{\mathcal{O}_v^{\times}}(c_v^2)=\mathbf{1}_{\mathcal{O}_v^{\times}}(c_v).
\end{align*}

Hence the contribution of this region is
\begin{multline}\label{e7.28}
G(\xi_v)\int_{\mathfrak{p}_v-\{0\}}
\int_{\mathcal{O}_v^{\times}}
\overline{\xi}_v\eta_v(c_v)|c_v|_v^{-1/2+s-it}dc_v\xi_v(y_v)|y_v|_v^{1/2+s+it}d^{\times}y_v\\
=\xi_v(\varpi_v)G(\xi_v)q_v^{-1/2-s-it}L_v(1/2+s+it,\xi_v)\mathbf{1}_{r_{\xi_v}=0}.
\end{multline}  

\item It remains to consider $y_v\in \mathcal{O}_v^{\times}$ and $c_v\in \mathcal{O}_v^{\times}$. The corresponding contribution is
\begin{align*}
\mathcal{I}(\mu_v,\xi_v):=\frac{G(\xi_v)}{\zeta_{F_v}(1)}\int_{\mathcal{O}_v^{\times}}
\int_{\mathcal{O}_v^{\times}}
\mu_v(c_v^2-y_v^2\tau_v)\mathbf{1}_{\mathcal{O}_v^{\times}}(c_v^2-y_v^2\tau_v)
\overline{\xi}_v\eta_v(c_v)\xi_v(y_v)d^{\times}c_vd^{\times}y_v.
\end{align*}

In view of \eqref{7.25}, we may assume $r_{\xi_v}\leq 1$. Let $\mathbb{F}_v^{\times}:=\mathcal{O}_v^{\times}/(1+\mathfrak{p}_v)$. Since the conductor exponents of $\mu_v$ and $\xi_v$ are at most $1$, reduction modulo $1+\mathfrak{p}_v$ gives  
\begin{equation}\label{7.29}
\mathcal{I}(\mu_v,\xi_v)=\frac{G(\xi_v)}{q_v^2\zeta_{F_v}(1)}\sum_{\substack{\alpha,\, \beta\in \mathbb{F}_v^{\times}\\
\alpha^2\neq \beta^2\tau_v}}
\mu_v(\alpha^2-\beta^2\tau_v)
\overline{\xi}_v\eta_v(\alpha)\xi_v(\beta).
\end{equation}

Making the change of variables $\alpha\mapsto \beta\alpha$ and using that $\eta_v$ is unramified, we simplify \eqref{7.29} to 
\begin{equation}\label{7.30}
\mathcal{I}(\mu_v,\xi_v)=\frac{(q_v-1)G(\xi_v)}{q_v^2\zeta_{F_v}(1)}\mathcal{J}(\mu_v,\xi_v)=q_v^{-1}(1-q_v^{-1})^2G(\xi_v)\mathcal{J}(\mu_v,\xi_v).
\end{equation}
where
\begin{align*}
\mathcal{J}(\mu_v,\xi_v):=\sum_{\alpha\in \mathbb{F}_v^{\times}:\ 
\alpha^2\neq \tau_v}
\mu_v(\alpha^2-\tau_v)
\overline{\xi}_v\eta_v(\alpha).
\end{align*}

Since $\mu_v$ has conductor exponent $1$, $\mathcal{J}(\mu_v,\xi_v)$ is a nontrivial multiplicative character sum on $\mathbb{F}_v^{\times}$. Equivalently, its summand is the trace function of a nontrivial rank-one Kummer sheaf on $\mathbb{A}^1_{\mathbb{F}_v}$ with bounded conductor. Weil's bound therefore gives
\begin{equation}\label{7.31}
\mathcal{J}(\mu_v,\xi_v)
\ll q_v^{1/2}.
\end{equation}
\end{itemize}

The bound \eqref{f7.21} now follows from \eqref{7.26}, \eqref{eq7.28}, \eqref{7.27}, \eqref{e7.28}, \eqref{7.30} and \eqref{7.31}.
\end{proof}

\subsection{Proof of Proposition \ref{prop7.1}}
Let $\pi=|\cdot|^s\boxplus \eta|\cdot|^{-s}$, where $|\Re(s)|\leq \delta$ with $0<\delta<1/2$. Substituting \eqref{7.2} into \eqref{7.1} gives 
\begin{multline}\label{7.40}
J_{\mathrm{Dual}}^{\heartsuit}(h(\cdot,s))=\frac{\Gamma_{\mathbb C}(k+1+s)^{d}}{2\pi\underset{\lambda=1}{\Res}\ \zeta_F(\lambda)}\sum_{\xi\in \widehat{C}_F}\int_{\mathbb{R}}L(1/2+it,\pi\otimes\xi)\\
\prod_{v\in \Sigma_{F,\infty}}\frac{\mathcal{P}_v(s,it;\xi_v)}{\Gamma_{\mathbb C}(k+1+s)}\prod_{v\in \Sigma_{F,\fin}}\frac{\mathcal{P}_v(s,it;\xi_v)}{L_v(1/2+it,\pi_v\otimes\xi_v)}dt.
\end{multline}

By Proposition \ref{prop7.3},
\begin{equation}\label{e7.41}
\prod_{v\in \Sigma_{F,\infty}}\frac{\mathcal{P}_v(s,it;\xi_v)}{\Gamma_{\mathbb C}(k+1+s)}\ll_{m,\delta} |\tau_{\infty}|_{\infty}^{-1/4-\Re(s)/2}\Omega_{m}(s,k;\xi_{\infty}|\cdot|_{\infty}^{it}),
\end{equation}
where 
\begin{align*}
\Omega_{m}(s,k;\xi_{\infty}|\cdot|_{\infty}^{it}):=\prod_{v\in \Sigma_{F,\infty}}\frac{Q^{-1/4+\delta}\mathbf{1}_{|\boldsymbol{\nu}_{\xi_v,t}|\leq 2Q}+(1+|\boldsymbol{\nu}_{\xi_v,t}|)^{-m}\mathbf{1}_{|\boldsymbol{\nu}_{\xi_v,t}|>2Q}}{(1+|\boldsymbol{\nu}_{\xi_v,t}-\Im(s)|)^{\Re(s)}}.
\end{align*}
Using \eqref{6.22} to evaluate $|\tau_{\infty}|_{\infty}$ in \eqref{e7.41}, we obtain  
\begin{equation}\label{7.41}
\prod_{v\in \Sigma_{F,\infty}}\frac{\mathcal{P}_v(s,it;\xi_v)}{\Gamma_{\mathbb C}(k+1+s)}\ll_{m,\delta} D_{E/F}^{-\frac{1+2\Re(s)}{4}}N_F(\mathfrak{a})^{-\frac{1}{2}-\Re(s)}\Omega_{m}(s,k;\xi_{\infty}|\cdot|_{\infty}^{it}).
\end{equation}  

Combining \eqref{6.5}, \eqref{fc7.21}, and Proposition \ref{prop7.5}, we also have  
\begin{equation}\label{7.42}
\prod_{v\in \Sigma_{F,\fin}}\frac{\mathcal{P}_v(s,it;\xi_v)}{L_v(1/2+it,\pi_v\otimes\xi_v)}\ll_{F,\delta,\varepsilon} D_{E/F}^{\frac{1}{2}+\delta}N_F(\mathfrak{a})^{\frac{1}{2}+\Re(s)}\sum_{\mathfrak{m}\supseteq \mathfrak{f}
}\frac{\mathbf{1}_{\xi\in \widehat{C}_F(\mathfrak{m})}}{N_F(\mathfrak{m})^{1/2-\delta}}. 
\end{equation}
 
Substituting \eqref{7.41} and \eqref{7.42} into \eqref{7.40} yields 
\begin{multline}\label{7.44}
J_{\mathrm{Dual}}^{\heartsuit}(h(\cdot,s))\ll_{F,m,\delta,\varepsilon} |\Gamma_{\mathbb C}(k+1+s)|^{d}D_{E/F}^{1/4+2\delta}\sum_{\mathfrak{m}\supseteq \mathfrak{f}}N_F(\mathfrak{m})^{-1/2+\delta}\\
\sum_{\xi\in \widehat{C}_F(\mathfrak{m})}\int_{\mathbb{R}}\big|L(1/2+it,\xi)L(1/2+it,\eta\xi)\big|\cdot \Omega_{m}(s,k;\xi_{\infty}|\cdot|_{\infty}^{it})dt.
\end{multline}
In particular, the implied constant is independent of $\mathfrak{a}$.

Let $M=D_{E/F}N_F(\mathfrak{f})$. From the definition of $\Omega_{m}(s,k;\xi_{\infty}|\cdot|_{\infty}^{it})$, we have
\begin{multline}\label{e7.44}
\Omega_{m}(s,k;\xi_{\infty}|\cdot|_{\infty}^{it})\ll_{F,m,\delta,\varepsilon}Q^{-(1/4-2\delta)d}M^{\varepsilon\delta d}\mathbf{1}_{|\boldsymbol{\nu}_{\xi_v,t}|\leq QM^{\varepsilon},\ \forall\, v\in \Sigma_{F,\infty}}\\
+(QM^{\varepsilon})^{-\frac{m}{2}}\sum_{\emptyset\neq S\subseteq \Sigma_{F,\infty}}\prod_{v\in S}|\boldsymbol{\nu}_{\xi_v,t}|^{\delta-\frac{m}{2}}\mathbf{1}_{|\boldsymbol{\nu}_{\xi_v,t}|\geq QM^{\varepsilon}}\prod_{w\in \Sigma_{F,\infty}-S}(QM^{\varepsilon})^{\delta}\mathbf{1}_{|\boldsymbol{\nu}_{\xi_w,t}|\leq QM^{\varepsilon}}.
\end{multline}

Substituting \eqref{e7.44} into \eqref{7.44} and applying the convexity bound for $L$-functions, 
\begin{multline}\label{7.45}
J_{\mathrm{Dual}}^{\heartsuit}(h(\cdot,s))\ll_{F,m,\delta,\varepsilon} \frac{D_{E/F}^{\frac{1}{4}+2\delta}|\Gamma_{\mathbb C}(k+1+s)|^{d}}{\big(Q^{\frac{1}{4}-2\delta}M^{-\varepsilon\delta}\big)^{d}}\mathcal{J}_{\mathrm{Split}}(s,\eta;\mathfrak{f})\\
+|\Gamma_{\mathbb C}(k+1+s)|^{d}(QM^{\varepsilon})^{-\frac{m}{2}+\delta d+10}\mathcal{E}(s,\eta;\mathfrak{f}),
\end{multline}
where $\mathcal{J}_{\mathrm{Split}}(s,\eta;\mathfrak{f})$ is defined by \eqref{e7.2} and $\mathcal{E}(s,\eta;\mathfrak{f})$ is defined by 
\begin{align*}
\sum_{\mathfrak{m}\supseteq \mathfrak{f}}\sum_{\emptyset\neq S\subseteq \Sigma_{F,\infty}}\sum_{\xi\in \widehat{C}_F(\mathfrak{m})}\int_{\mathbb{R}}\prod_{v\in S}|\boldsymbol{\nu}_{\xi_v,t}|^{\frac{1-m}{2}+10}\mathbf{1}_{|\boldsymbol{\nu}_{\xi_v,t}|\geq QM^{\varepsilon}}
\prod_{w\in \Sigma_{F,\infty}-S}\mathbf{1}_{|\boldsymbol{\nu}_{\xi_w,t}|\leq QM^{\varepsilon}}dt.
\end{align*}

We decompose both the characters $\xi$ and the $t$-ranges according to dyadic intervals for the quantities $|\boldsymbol{\nu}_{\xi_v,t}|$. Taking $m$ sufficiently large, for instance $m=10^{4}\floor{\varepsilon^{-2}}$, gives
\begin{equation}\label{7.47}
(QM^{\varepsilon})^{-\frac{m}{2}+\delta d+10}\mathcal{E}(s,\eta;\mathfrak{f})\ll_{F,\delta,\varepsilon} Q^{-100}N_F(\mathfrak{f}\mathfrak{D}_{E})^{-100}.
\end{equation}

Therefore, \eqref{eq7.3} follows from \eqref{7.45} and \eqref{7.47}. 

Now we prove \eqref{eq7.4}. By the Burgess bound \cite[Corollary 1.2]{Yan26}, 
\begin{multline}\label{7.51}
\mathcal{J}_{\mathrm{Split}}(s,\eta;\mathfrak{f})\ll_{F,\varepsilon} (QN_F(\mathfrak{f}\mathfrak{D}_{E})^{\varepsilon})^{(3/16+\varepsilon)d}\sum_{\mathfrak{m}\supseteq \mathfrak{f}}N_F(\mathfrak{m})^{-1/2+\delta}\\
N_F(\mathfrak{m}\mathfrak{D}_{E/F})^{3/16+\varepsilon}\sum_{\xi\in \widehat{C}_F(\mathfrak{m})}\int_{\mathbf{B}(\boldsymbol{\nu}_{\xi_{\infty},t},QN_F(\mathfrak{f}\mathfrak{D}_{E})^{\varepsilon})}\big|L(1/2+it,\xi)\big|dt.
\end{multline} 

We use the following standard mean-value estimate for Hecke $L$-functions. Let $R\geq 1$. Then, for any $\varepsilon>0$,
\begin{equation}\label{7.50}
\sum_{\xi\in \widehat{C}_F(\mathfrak{m})}
\int_{\mathbf{B}(\boldsymbol{\nu}_{\xi_{\infty},t},R)}
\big|L(1/2+it,\xi)\big|^2dt
\ll_{F,\varepsilon}
\left(N_F(\mathfrak{m})R^{d}\right)^{1+\varepsilon}.
\end{equation}
This is the Lindel\"{o}f-on-average bound for Hecke characters of conductor $\mathfrak{m}$ and archimedean parameters bounded by $R$. It follows from the approximate functional equation, which reduces the central values to Dirichlet polynomials of length
$\ll (N_F(\mathfrak{m})R^{d})^{1/2+\varepsilon}$, together with the large sieve inequality for Hecke characters over $F$.

Applying the Cauchy--Schwarz inequality to the innermost sum and integral in \eqref{7.51}, and then using \eqref{7.50} with $R=QN_F(\mathfrak{f}\mathfrak{D}_{E})^{\varepsilon}$, gives
\begin{multline}\label{7.52}
\mathcal{J}_{\mathrm{Split}}(s,\eta;\mathfrak{f})\ll_{F,\varepsilon} (QN_F(\mathfrak{f}\mathfrak{D}_{E})^{\varepsilon})^{(3/16+\varepsilon)d}\sum_{\mathfrak{m}\supseteq \mathfrak{f}}N_F(\mathfrak{m})^{-1/2+\delta}\\
N_F(\mathfrak{m}\mathfrak{D}_{E/F})^{\frac{3}{16}+\varepsilon}N_F(\mathfrak{m})^{1+\varepsilon}Q^{d}.
\end{multline}

Therefore, \eqref{eq7.4} follows from \eqref{eq7.3} and \eqref{7.52}.

\section{The First Moment and Nonvanishing}\label{sec8}

Let $0<\delta<1/2$ and $|\Re(s)|\leq \delta$.  Combining Propositions
\ref{prop3.2}, \ref{prop6.1}, \ref{prop6.2}, and \ref{prop7.1} with Theorem
\ref{thm3.1}, and then dividing both sides by
$2^{-d}\pi^{d}\Gamma_{\mathbb C}(k+1+s)^{d}$,
we obtain
\begin{multline}\label{e8.1}
\sum_{\chi\in X_{E/F}^{\mathrm{un}}(k,\boldsymbol{\Phi})}\frac{L(1/2+s,\widetilde{\mu}\chi)}{L(1,\eta)}
=\pi^{-d}D_{E/F}^{1/2}D_F^{1/2}L^{(\mathfrak{f})}(1+2s,\eta)\\
+\frac{\epsilon_{\mu,k}\pi^{-d}
D_{E/F}^{1/2-2s}D_F^{1/2-2s}}{2^{2sd}
N_F(\mathfrak{f})^{2s}
}L^{(\mathfrak{f})}(1-2s,\eta)+L(1,\eta)^{-1}\mathcal{E}_{\mu}(s,k).
\end{multline}
Here $\mathcal{E}_{\mu}(s,k)$ is a holomorphic function satisfying \eqref{e1.3}. 

Theorem \ref{thma} follows from \eqref{e8.1} after rewriting the two main terms
in terms of $\mathcal{M}_{\mathfrak{f}}(s,k)$ and
$\mathcal{M}_{\mathfrak{f}}(-s,k)$, as defined in \eqref{a}.
Corollaries \ref{cor1.4} and \ref{cor1.6} then follow from Theorem \ref{thma},
the class number formula
\begin{align*}
h_{E/F}
=2(2\pi)^{-d}D_{E/F}^{1/2}D_F^{1/2}L(1,\eta),
\end{align*}
and the elementary calculation
\begin{multline*}
\left.\frac{d}{ds}\right|_{s=0}
\frac{\mathcal{M}_{\mathfrak{f}}(s,k)+\epsilon_{\mu,k}\mathcal{M}_{\mathfrak{f}}(-s,k)}
{2^{sd}D_{E/F}^{s}N_F(\mathfrak f)^{s}D_F^{s}}
=2\pi^{-d}D_{E/F}^{1/2}D_F^{1/2}
\Big[(1-\epsilon_{\mu,k})(L^{(\mathfrak f)})'(1,\eta)\\
-\epsilon_{\mu,k}L^{(\mathfrak f)}(1,\eta)
\log(2^{d}D_{E/F}D_FN_F(\mathfrak f))\Big].
\end{multline*}

Finally, by \cite[Corollary 1.7]{Yan26b}, we have
\begin{equation}\label{8.2}
L(1/2+it,\widetilde{\mu}\chi)\ll_{F,\varepsilon} N_F(\mathfrak{f}\mathfrak{D}_{E})^{\frac{1}{2}-\frac{1-2\vartheta}{12}+\varepsilon}(1+|t|)^{\frac{d}{4}-\frac{(1-2\theta)d}{24}+\varepsilon}.
\end{equation}
Theorem \ref{thmc} follows by combining Corollaries \ref{cor1.4} and
\ref{cor1.6} with \eqref{8.2}.  The restriction $d\geq 14$ is obtained by using
the worst available bound $\vartheta=7/64$.

\appendix

\section{Oscillatory Integral Estimates}\label{App}

In this appendix, we collect several elementary oscillatory integral estimates used in \textsection\ref{sec7.1.1}. Throughout, let
\begin{align*}
n:=2k+1,\quad Q':=\left(\Im(s)^2+n^2/4\right)^{1/2},
\end{align*}
and assume that $0<\delta<1/2$ and $|\Re(s)|\leq \delta$.

\subsection{A Change of Variables}\label{app:change-of-variables}

We use the change of variables $x=\log\tan\theta$. Then
\begin{align*}
\sin\theta=e^{x/2}(2\cosh x)^{-1/2},
\quad
\cos\theta=e^{-x/2}(2\cosh x)^{-1/2},
\quad
d\theta=(2\cosh x)^{-1}dx.
\end{align*}
After expanding $\sin(n\theta)$ and $\cos(n\theta)$ into exponentials, the integrals $\mathcal{J}_{k}^{\pm}(s,T)$ reduce to a linear combination of 
\begin{align*}
I_{\pm}:=\int_{\mathbb R}\beta_s(x)e^{i\Psi_{\pm}(x)}dx,
\end{align*}
where $\beta_s(x):=(2\cosh x)^{-1/2-\Re(s)}$, and
\begin{align*}
\Psi_{\pm}(x):=
-Tx-\Im(s)\log(2\cosh x)\pm n\arctan(e^x).
\end{align*}
We have
\begin{align*}
\Psi_{\pm}'(x)
&=
-T-\Im(s)\tanh x
\pm \frac{n}{2\cosh x},\\
\Psi_{\pm}''(x)
&=
-\Im(s)(\cosh x)^{-2}
\mp \frac{n}{2}\tanh x(\cosh x)^{-1}.
\end{align*}

Put $u:=\tanh x$ and $h:=(\cosh x)^{-1}$. Then $u^2+h^2=1$. Define
\begin{align*}
F_{\pm}(x):=-\Im(s)u\pm 2^{-1}nh,
\quad
B_{\pm}(x):=-\Im(s)h\mp 2^{-1}nu.
\end{align*}
Then
\begin{align*}
\Psi_{\pm}'(x)=F_{\pm}(x)-T,
\quad
\Psi_{\pm}''(x)=hB_{\pm}(x),
\end{align*}
and
\begin{align*}
F_{\pm}(x)^2+B_{\pm}(x)^2=Q'^2.
\end{align*}
Moreover,
\begin{align*}
\Psi_{\pm}'''(x)
=
-uhB_{\pm}(x)-h^2F_{\pm}(x).
\end{align*}

\subsection{Derivative Bounds for the Amplitude}\label{app:amplitude-bounds}

For every $r\geq 0$, one has
\begin{align*}
\beta_s^{(r)}(x)
=
(2\cosh x)^{-1/2-\Re(s)}P_{r,\Re(s)}(\tanh x),
\end{align*}
where $P_{r,\Re(s)}$ is a polynomial in $\tanh x$ whose degree and coefficients are bounded in terms of $r$ and $\delta$. Hence
\begin{align*}
\beta_s^{(r)}(x)
\ll_{r,\delta}
(\cosh x)^{-1/2+\delta}.
\end{align*}
In particular, if $j\leq |x|\leq j+1$ and $H_j:=e^{-j}$, then $\beta_s^{(r)}(x)\ll_{r,\delta}H_j^{1/2-\delta}$. 

\subsection{The Non-degenerate Stationary Range}
\begin{lemma}\label{lemA.1}
Let $0<\varepsilon<1$. Suppose that $|T|\leq (1-\varepsilon)Q'$. Then
\begin{align*}
I_{\pm}\ll_{\delta,\varepsilon}Q'^{-1/2+\delta}.
\end{align*}
\end{lemma}

\begin{proof}
Since $F_{\pm}(x)^2+B_{\pm}(x)^2=Q'^2$ and $|T|\leq (1-\varepsilon)Q'$, there exists $c=c(\varepsilon)>0$ such that, for every $x\in\mathbb R$,
\begin{align*}
|F_{\pm}(x)-T|\geq cQ'
\quad\text{or}\quad
|B_{\pm}(x)|\geq cQ'.
\end{align*}
Indeed, otherwise $F_{\pm}(x)$ would be within $cQ'$ of $T$ and $B_{\pm}(x)$ would be $O(cQ')$, contradicting $F_{\pm}(x)^2+B_{\pm}(x)^2=Q'^2$ if $c$ is chosen sufficiently small in terms of $\varepsilon$.

Decompose $\mathbb R$ into the intervals $j\leq |x|\leq j+1$, $j\geq 0$. On each such block, $h\asymp H_j:=e^{-j}$. After subdividing the block into $O_{\varepsilon}(1)$ subintervals, we may assume that either
\begin{align*}
|\Psi_{\pm}'(x)|\gg_{\varepsilon} Q'
\end{align*}
throughout, or
\begin{align*}
|\Psi_{\pm}''(x)|=h|B_{\pm}(x)|\gg_{\varepsilon} Q'H_j
\end{align*}
throughout.

On subintervals of the first type, integration by parts gives
\begin{align*}
\int \beta_s(x)e^{i\Psi_{\pm}(x)}dx
\ll_{\delta,\varepsilon}
H_j^{1/2-\delta}Q'^{-1}.
\end{align*}
On subintervals of the second type, van der Corput's second derivative estimate gives
\begin{align*}
\int \beta_s(x)e^{i\Psi_{\pm}(x)}dx
\ll_{\delta,\varepsilon}
H_j^{1/2-\delta}(Q'H_j)^{-1/2}.
\end{align*}
Together with the trivial estimate, we obtain
\begin{align*}
\int_{j\leq |x|\leq j+1}
\beta_s(x)e^{i\Psi_{\pm}(x)}dx
\ll_{\delta,\varepsilon}
H_j^{1/2-\delta}
\min\{1,(Q'H_j)^{-1/2}\}.
\end{align*}
Summing over $j\geq 0$ gives
\begin{align*}
I_{\pm}\ll_{\delta,\varepsilon}
\sum_{j\geq 0}
H_j^{1/2-\delta}
\min\{1,(Q'H_j)^{-1/2}\}\ll_{\delta,\varepsilon}
Q'^{-1/2+\delta}.
\end{align*}
This proves the claim.
\end{proof}

\subsection{The Cubic Transition Range}\begin{lemma}\label{lemA.2}
Suppose that $(1-\varepsilon)Q'\leq |T|\leq (1+\varepsilon)Q'$. Then
\begin{align*}
I_{\pm}\ll_{\delta,\varepsilon}Q'^{-1/4+\delta}.
\end{align*}
\end{lemma}
\begin{proof}
We use the same dyadic decomposition $j\leq |x|\leq j+1$, $j\geq 0$, and write $H_j=e^{-j}$. On this block, $h\asymp H_j$ and $\beta_s^{(r)}(x)\ll_{r,\delta}H_j^{1/2-\delta}$. 

We claim that, after subdividing the block into $O(1)$ subintervals, one of the following alternatives holds throughout each subinterval:
\begin{align}
|\Psi_{\pm}'(x)|&\gg Q', \label{A.1}\\
|\Psi_{\pm}''(x)|&\gg Q'H_j^2, \label{A.2}\\
|\Psi_{\pm}'''(x)|&\gg Q'H_j^2. \label{A.3}
\end{align}
Indeed, suppose that \eqref{A.1} and \eqref{A.2} both fail, with sufficiently small implicit constants. Since
\begin{align*}
\Psi_{\pm}'(x)=F_{\pm}(x)-T,
\quad
\Psi_{\pm}''(x)=hB_{\pm}(x),
\end{align*}
and $h\asymp H_j$, we have
\begin{align*}
|F_{\pm}(x)-T|\ll Q',
\quad
|B_{\pm}(x)|\ll Q'H_j.
\end{align*}
Choosing the implicit constant in \eqref{A.1} sufficiently small and using $|T|\geq (1-\varepsilon)Q'$, we get $|F_{\pm}(x)|\gg_{\varepsilon} Q'$. Therefore, from
\begin{align*}
\Psi_{\pm}'''(x)
=
-uhB_{\pm}(x)-h^2F_{\pm}(x),
\end{align*}
we get
\begin{align*}
|\Psi_{\pm}'''(x)|\gg_{\varepsilon} Q'H_j^2,
\end{align*}
after decreasing the implicit constants if necessary. This proves the pointwise trichotomy. The subdivision into $O(1)$ subintervals follows, as before, from the fact that after writing $u=\tanh x$ and $h=(1-u^2)^{1/2}$, the boundary equations have uniformly bounded algebraic complexity.

On subintervals satisfying \eqref{A.1}, integration by parts gives
\begin{align*}
\int \beta_s(x)e^{i\Psi_{\pm}(x)}dx
\ll_{\delta,\varepsilon}
H_j^{1/2-\delta}Q'^{-1}.
\end{align*}
On subintervals satisfying \eqref{A.2}, van der Corput's second derivative estimate gives
\begin{align*}
\int \beta_s(x)e^{i\Psi_{\pm}(x)}dx
\ll_{\delta,\varepsilon}
H_j^{1/2-\delta}(Q'H_j^2)^{-1/2}.
\end{align*}
On subintervals satisfying \eqref{A.3}, van der Corput's third derivative estimate gives
\begin{align*}
\int \beta_s(x)e^{i\Psi_{\pm}(x)}dx
\ll_{\delta,\varepsilon}
H_j^{1/2-\delta}(Q'H_j^2)^{-1/3}.
\end{align*}
Combining these bounds with the trivial estimate yields
\begin{align*}
\int_{j\leq |x|\leq j+1}
\beta_s(x)e^{i\Psi_{\pm}(x)}dx
\ll_{\delta,\varepsilon}
H_j^{1/2-\delta}
\min\{1,(Q'H_j^2)^{-1/3}\}.
\end{align*}
Summing over $j\geq 0$, we obtain
\begin{align*}
I_{\pm}
&\ll_{\delta,\varepsilon}
\sum_{j\geq 0}
H_j^{1/2-\delta}
\min\{1,(Q'H_j^2)^{-1/3}\}\\
&\ll_{\delta,\varepsilon}
Q'^{-1/3}
\sum_{H_j\geq Q'^{-1/2}}H_j^{-1/6-\delta}
+
\sum_{H_j<Q'^{-1/2}}H_j^{1/2-\delta}\ll_{\delta,\varepsilon}
Q'^{-1/4+\delta}.
\end{align*}
This proves the claim.
\end{proof}

\subsection{The Non-stationary Range}
\begin{lemma}\label{lemA.3}
Let $0<\varepsilon<1$. Suppose that $|T|\geq (1+\varepsilon)Q'$. Then, for every $m\geq 0$,
\begin{align*}
I_{\pm}\ll_{m,\delta,\varepsilon}Q'^{-m}.
\end{align*}
\end{lemma}

\begin{proof}
Since $F_{\pm}(x)^2+B_{\pm}(x)^2=Q'^2$, we have $|F_{\pm}(x)|\leq Q'$ for every $x\in\mathbb R$. Therefore, if $|T|\geq (1+\varepsilon)Q'$, then
\begin{align*}
|\Psi_{\pm}'(x)|=|F_{\pm}(x)-T|
\geq |T|-|F_{\pm}(x)|
\geq \varepsilon Q'.
\end{align*}
Thus the phase has no stationary point in this range.

We again decompose $\mathbb R$ into the intervals $j\leq |x|\leq j+1$, $j\geq 0$, 
and write $H_j=e^{-j}$. On such a block we have
\begin{align*}
h\asymp H_j,\quad
\beta_s^{(r)}(x)\ll_{r,\delta}H_j^{1/2-\delta}.
\end{align*}
Moreover, for every $r\geq 2$,
\begin{align*}
\Psi_{\pm}^{(r)}(x)\ll_r Q'H_j.
\end{align*}
Indeed, every derivative of $\tanh x$ or $(\cosh x)^{-1}$ of positive order is $O_r(h)$, and hence is $O_r(H_j)$ on the block.

Let $J$ be one of these blocks. We integrate by parts repeatedly using the operator
\begin{align*}
\mathcal{L}:=\frac{1}{i\Psi_{\pm}'(x)}\frac{d}{dx},
\quad
\mathcal{L}e^{i\Psi_{\pm}(x)}=e^{i\Psi_{\pm}(x)}.
\end{align*}
After integrating by parts $m+1$ times, the derivative bounds above, together with $|\Psi_{\pm}'(x)|\gg_{\varepsilon} Q'$, give
\begin{align*}
\int_J \beta_s(x)e^{i\Psi_{\pm}(x)}dx
\ll_{m,\delta,\varepsilon}
H_j^{1/2-\delta}Q'^{-m-1}.
\end{align*}
Here the derivatives falling on $(\Psi_{\pm}'(x))^{-1}$ are harmless, since
\begin{align*}
\frac{\Psi_{\pm}^{(r)}(x)}{\Psi_{\pm}'(x)}
\ll_{r,\varepsilon}H_j
\end{align*}
for every $r\geq 2$.

Summing over $j\geq 0$, we obtain
\begin{align*}
I_{\pm}\ll_{m,\delta,\varepsilon}
Q'^{-m-1}
\sum_{j\geq 0}H_j^{1/2-\delta}\ll_{m,\delta,\varepsilon}Q'^{-m-1}.
\end{align*}
Since $Q'\gg 1$, this implies $I_{\pm}\ll_{m,\delta,\varepsilon}Q'^{-m}$. This proves the lemma.
\end{proof}

\bibliographystyle{alpha}

\bibliography{LY}

@book {Gro80,
    AUTHOR = {Gross, Benedict H.},
     TITLE = {Arithmetic on elliptic curves with complex multiplication},
    SERIES = {Lecture Notes in Mathematics},
    VOLUME = {776},
      NOTE = {With an appendix by B. Mazur},
 PUBLISHER = {Springer, Berlin},
      YEAR = {1980},
     PAGES = {iii+95},
      ISBN = {3-540-09743-0},
   MRCLASS = {10D25 (14K22)},
  MRNUMBER = {563921},
MRREVIEWER = {A.\ N.\ Andrianov},
}

@article{HPY25,
  title={The cubic moment of {$L$}-functions for specified local component families},
  author={Hu, Yueke and Petrow, Ian and Young, Matthew P},
  journal={arXiv preprint arXiv:2506.14741},
  year={2025}
}

@article{HY26,
  title={{R}ankin--{S}elberg Subconvexity via Spectral Reciprocity},
  author={Humphries, Peter and Yang, Liyang},
  journal={arXiv preprint arXiv:2606.11451},
  year={2026}
}

@article {IKM19,
    AUTHOR = {Iannuzzi, Arianna and Kim, Byoung Du and Masri, Riad and
              Mathers, Alexander and Ross, Maria and Tsai, Wei-Lun},
     TITLE = {Nonvanishing of {H}ecke {$L$}-functions and {B}loch-{K}ato
              {$p$}-{S}elmer groups},
   JOURNAL = {Math. Res. Lett.},
  FJOURNAL = {Mathematical Research Letters},
    VOLUME = {26},
      YEAR = {2019},
    NUMBER = {4},
     PAGES = {1145--1177},
      ISSN = {1073-2780,1945-001X},
   MRCLASS = {11F67 (11R11 11R23)},
  MRNUMBER = {4028114},
MRREVIEWER = {Neil\ P.\ Dummigan},
       DOI = {10.4310/MRL.2019.v26.n4.a8},
       URL = {https://doi.org/10.4310/MRL.2019.v26.n4.a8},
}

@article {KL89,
    AUTHOR = {Kolyvagin, V. A. and Logach\"ev, D. Yu.},
     TITLE = {Finiteness of the {S}hafarevich-{T}ate group and the group of
              rational points for some modular abelian varieties},
   JOURNAL = {Algebra i Analiz},
  FJOURNAL = {Algebra i Analiz},
    VOLUME = {1},
      YEAR = {1989},
    NUMBER = {5},
     PAGES = {171--196},
      ISSN = {0234-0852},
   MRCLASS = {11G10 (11G40 14K15)},
  MRNUMBER = {1036843},
MRREVIEWER = {Takeshi\ Ooe},
}

@article {KMY11,
    AUTHOR = {Kim, Byoung Du and Masri, Riad and Yang, Tong Hai},
     TITLE = {Nonvanishing of {H}ecke {$L$}-functions and the {B}loch-{K}ato
              conjecture},
   JOURNAL = {Math. Ann.},
  FJOURNAL = {Mathematische Annalen},
    VOLUME = {349},
      YEAR = {2011},
    NUMBER = {2},
     PAGES = {301--343},
      ISSN = {0025-5831,1432-1807},
   MRCLASS = {11R42 (11R47)},
  MRNUMBER = {2753824},
MRREVIEWER = {Matthew\ P.\ Young},
       DOI = {10.1007/s00208-010-0521-7},
       URL = {https://doi.org/10.1007/s00208-010-0521-7},
}

@article {LX04,
    AUTHOR = {Liu, Chunlei and Xu, Lanju},
     TITLE = {The vanishing order of certain {H}ecke {$L$}-functions of
              imaginary quadratic fields},
   JOURNAL = {J. Number Theory},
  FJOURNAL = {Journal of Number Theory},
    VOLUME = {108},
      YEAR = {2004},
    NUMBER = {1},
     PAGES = {76--89},
      ISSN = {0022-314X,1096-1658},
   MRCLASS = {11R47 (11F67 11R42)},
  MRNUMBER = {2078658},
MRREVIEWER = {Stephen\ D.\ Miller},
       DOI = {10.1016/j.jnt.2004.04.007},
       URL = {https://doi.org/10.1016/j.jnt.2004.04.007},
}

@article {Mas07,
    AUTHOR = {Masri, Riad},
     TITLE = {Asymptotics for sums of central values of canonical {H}ecke
              {$L$}-series},
   JOURNAL = {Int. Math. Res. Not. IMRN},
  FJOURNAL = {International Mathematics Research Notices. IMRN},
      YEAR = {2007},
    NUMBER = {19},
     PAGES = {Art. ID rnm065, 27},
      ISSN = {1073-7928,1687-0247},
   MRCLASS = {11F67 (11F37 11G15 11M06)},
  MRNUMBER = {2359540},
MRREVIEWER = {J.\ C.\ Lagarias},
       DOI = {10.1093/imrn/rnm065},
       URL = {https://doi.org/10.1093/imrn/rnm065},
}

@article {Mas07b,
    AUTHOR = {Masri, Riad},
     TITLE = {Quantitative nonvanishing of {$L$}-series associated to
              canonical {H}ecke characters},
   JOURNAL = {Int. Math. Res. Not. IMRN},
  FJOURNAL = {International Mathematics Research Notices. IMRN},
      YEAR = {2007},
    NUMBER = {19},
     PAGES = {Art. ID rnm070, 16},
      ISSN = {1073-7928,1687-0247},
   MRCLASS = {11R47 (11M41 11R42)},
  MRNUMBER = {2359543},
MRREVIEWER = {Robert\ Perlis},
       DOI = {10.1093/imrn/rnm070},
       URL = {https://doi.org/10.1093/imrn/rnm070},
}

@article {MR82,
    AUTHOR = {Montgomery, Hugh L. and Rohrlich, David E.},
     TITLE = {On the {$L$}-functions of canonical {H}ecke characters of
              imaginary quadratic fields. {II}},
   JOURNAL = {Duke Math. J.},
  FJOURNAL = {Duke Mathematical Journal},
    VOLUME = {49},
      YEAR = {1982},
    NUMBER = {4},
     PAGES = {937--942},
      ISSN = {0012-7094,1547-7398},
   MRCLASS = {12A70 (10D24 14K07)},
  MRNUMBER = {683009},
MRREVIEWER = {Reinhard\ B\"olling},
       URL = {http://projecteuclid.org/euclid.dmj/1077315537},
}

@article {MV10,
	AUTHOR = {Michel, Philippe and Venkatesh, Akshay},
	TITLE = {The subconvexity problem for {${\rm GL}_2$}},
	JOURNAL = {Publ. Math. Inst. Hautes \'{E}tudes Sci.},
	FJOURNAL = {Publications Math\'{e}matiques. Institut de Hautes \'{E}tudes
		Scientifiques},
	NUMBER = {111},
	YEAR = {2010},
	PAGES = {171--271},
	ISSN = {0073-8301},
	MRCLASS = {11F67 (11F37 11F70 22E55 58J51)},
	MRNUMBER = {2653249},
	MRREVIEWER = {Andre Reznikov},
	DOI = {10.1007/s10240-010-0025-8},
	URL = {https://doi.org/10.1007/s10240-010-0025-8},
}

@article {MY00,
    AUTHOR = {Miller, Stephen D. and Yang, Tonghai},
     TITLE = {Non-vanishing of the central derivative of canonical {H}ecke
              {$L$}-functions},
   JOURNAL = {Math. Res. Lett.},
  FJOURNAL = {Mathematical Research Letters},
    VOLUME = {7},
      YEAR = {2000},
    NUMBER = {2-3},
     PAGES = {263--277},
      ISSN = {1073-2780},
   MRCLASS = {11F67 (11G05)},
  MRNUMBER = {1764321},
MRREVIEWER = {Andrea\ Mori},
       DOI = {10.4310/MRL.2000.v7.n3.a2},
       URL = {https://doi.org/10.4310/MRL.2000.v7.n3.a2},
}

@article {MY11,
    AUTHOR = {Masri, Riad and Yang, Tonghai},
     TITLE = {Nonvanishing of {H}ecke {$L$}-functions for {CM} fields and
              ranks of abelian varieties},
   JOURNAL = {Geom. Funct. Anal.},
  FJOURNAL = {Geometric and Functional Analysis},
    VOLUME = {21},
      YEAR = {2011},
    NUMBER = {3},
     PAGES = {648--679},
      ISSN = {1016-443X,1420-8970},
   MRCLASS = {11F67 (11F27 11F41 11G10 11G40 11R80)},
  MRNUMBER = {2810860},
MRREVIEWER = {Andrea\ Mori},
       DOI = {10.1007/s00039-011-0121-z},
       URL = {https://doi.org/10.1007/s00039-011-0121-z},
}

@incollection {Nek07,
    AUTHOR = {Nekov\'a\v r, Jan},
     TITLE = {The {E}uler system method for {CM} points on {S}himura curves},
 BOOKTITLE = {{$L$}-functions and {G}alois representations},
    SERIES = {London Math. Soc. Lecture Note Ser.},
    VOLUME = {320},
     PAGES = {471--547},
 PUBLISHER = {Cambridge Univ. Press, Cambridge},
      YEAR = {2007},
      ISBN = {978-0-521-69415-5},
   MRCLASS = {11G15 (11F80 11G18)},
  MRNUMBER = {2392363},
MRREVIEWER = {Benjamin\ V.\ Howard},
       DOI = {10.1017/CBO9780511721267.014},
       URL = {https://doi.org/10.1017/CBO9780511721267.014},
}

@article {Roh80a,
    AUTHOR = {Rohrlich, David E.},
     TITLE = {The nonvanishing of certain {H}ecke {$L$}-functions at the
              center of the critical strip},
   JOURNAL = {Duke Math. J.},
  FJOURNAL = {Duke Mathematical Journal},
    VOLUME = {47},
      YEAR = {1980},
    NUMBER = {1},
     PAGES = {223--232},
      ISSN = {0012-7094,1547-7398},
   MRCLASS = {12A70 (10D24 14K07)},
  MRNUMBER = {563377},
MRREVIEWER = {Reinhard\ B\"olling},
       URL = {http://projecteuclid.org/euclid.dmj/1077313872},
}

@article {Roh80b,
    AUTHOR = {Rohrlich, David E.},
     TITLE = {On the {$L$}-functions of canonical {H}ecke characters of
              imaginary quadratic fields},
   JOURNAL = {Duke Math. J.},
  FJOURNAL = {Duke Mathematical Journal},
    VOLUME = {47},
      YEAR = {1980},
    NUMBER = {3},
     PAGES = {547--557},
      ISSN = {0012-7094,1547-7398},
   MRCLASS = {12A70 (10H10)},
  MRNUMBER = {587165},
MRREVIEWER = {Reinhard\ B\"olling},
       URL = {http://projecteuclid.org/euclid.dmj/1077314180},
}

@article {Roh80,
    AUTHOR = {Rohrlich, David E.},
     TITLE = {Galois conjugacy of unramified twists of {H}ecke characters},
   JOURNAL = {Duke Math. J.},
  FJOURNAL = {Duke Mathematical Journal},
    VOLUME = {47},
      YEAR = {1980},
    NUMBER = {3},
     PAGES = {695--703},
      ISSN = {0012-7094,1547-7398},
   MRCLASS = {12A70 (12A65)},
  MRNUMBER = {587174},
MRREVIEWER = {Alan\ Candiotti},
       URL = {http://projecteuclid.org/euclid.dmj/1077314189},
}

@article {Roh82,
    AUTHOR = {Rohrlich, David E.},
     TITLE = {Root numbers of {H}ecke {$L$}-functions of {CM} fields},
   JOURNAL = {Amer. J. Math.},
  FJOURNAL = {American Journal of Mathematics},
    VOLUME = {104},
      YEAR = {1982},
    NUMBER = {3},
     PAGES = {517--543},
      ISSN = {0002-9327,1080-6377},
   MRCLASS = {12A70 (10H10)},
  MRNUMBER = {658544},
MRREVIEWER = {Kenneth\ Kramer},
       DOI = {10.2307/2374152},
       URL = {https://doi.org/10.2307/2374152},
}

@article {Rub81,
    AUTHOR = {Rubin, Karl},
     TITLE = {Elliptic curves with complex multiplication and the conjecture
              of {B}irch and {S}winnerton-{D}yer},
   JOURNAL = {Invent. Math.},
  FJOURNAL = {Inventiones Mathematicae},
    VOLUME = {64},
      YEAR = {1981},
    NUMBER = {3},
     PAGES = {455--470},
      ISSN = {0020-9910,1432-1297},
   MRCLASS = {10D25 (12B30)},
  MRNUMBER = {632985},
MRREVIEWER = {Sheldon\ Kamienny},
       DOI = {10.1007/BF01389277},
       URL = {https://doi.org/10.1007/BF01389277},
}

@article {RY99,
    AUTHOR = {Rodriguez Villegas, Fernando and Yang, Tonghai},
     TITLE = {Central values of {H}ecke {$L$}-functions of {CM} number
              fields},
   JOURNAL = {Duke Math. J.},
  FJOURNAL = {Duke Mathematical Journal},
    VOLUME = {98},
      YEAR = {1999},
    NUMBER = {3},
     PAGES = {541--564},
      ISSN = {0012-7094,1547-7398},
   MRCLASS = {11F67 (11F37 11F41 11G40)},
  MRNUMBER = {1695801},
MRREVIEWER = {Andrea\ Mori},
       DOI = {10.1215/S0012-7094-99-09817-4},
       URL = {https://doi.org/10.1215/S0012-7094-99-09817-4},
}

@article{Shi76,
  title={The special values of the zeta functions associated with cusp forms},
  author={Shimura, Goro},
  journal={Communications on pure and applied Mathematics},
  volume={29},
  number={6},
  pages={783--804},
  year={1976},
  publisher={Wiley Online Library}
}

@article {Sta74,
    AUTHOR = {Stark, H. M.},
     TITLE = {Some effective cases of the {B}rauer-{S}iegel theorem},
   JOURNAL = {Invent. Math.},
  FJOURNAL = {Inventiones Mathematicae},
    VOLUME = {23},
      YEAR = {1974},
     PAGES = {135--152},
      ISSN = {0020-9910,1432-1297},
   MRCLASS = {10H10 (12A50 12A70)},
  MRNUMBER = {342472},
MRREVIEWER = {W.\ Narkiewicz},
       DOI = {10.1007/BF01405166},
       URL = {https://doi.org/10.1007/BF01405166},
}

@article {Tem10,
    AUTHOR = {Templier, Nicolas},
     TITLE = {On asymptotic values of canonical quadratic {$L$}-functions},
   JOURNAL = {Int. J. Number Theory},
  FJOURNAL = {International Journal of Number Theory},
    VOLUME = {6},
      YEAR = {2010},
    NUMBER = {8},
     PAGES = {1717--1730},
      ISSN = {1793-0421,1793-7310},
   MRCLASS = {11M41 (11F67 11L40 11R42)},
  MRNUMBER = {2755467},
MRREVIEWER = {Robert\ C.\ Rhoades},
       DOI = {10.1142/S1793042110003678},
       URL = {https://doi.org/10.1142/S1793042110003678},
}

@unpublished{TZ08,
  author = {Tian, Ye and Zhang, Shou-Wu},
  title = {Kolyvagin systems of {CM} points on {Shimura} curves},
  note = {Preprint},
  year = {2008}
}

@article {Ven10,
    AUTHOR = {Venkatesh, Akshay},
     TITLE = {Sparse equidistribution problems, period bounds and
              subconvexity},
   JOURNAL = {Ann. of Math. (2)},
  FJOURNAL = {Annals of Mathematics. Second Series},
    VOLUME = {172},
      YEAR = {2010},
    NUMBER = {2},
     PAGES = {989--1094},
      ISSN = {0003-486X,1939-8980},
   MRCLASS = {11F67 (11F70 11M41 58J51)},
  MRNUMBER = {2680486},
MRREVIEWER = {Philippe\ G.\ Michel},
       DOI = {10.4007/annals.2010.172.989},
       URL = {https://doi.org/10.4007/annals.2010.172.989},
}

@article {Yan20,
    AUTHOR = {Yang, Liyang},
     TITLE = {An explicit {CM} type norm formula and effective nonvanishing
              of class group {$L$}-functions for {CM} fields},
   JOURNAL = {Pacific J. Math.},
  FJOURNAL = {Pacific Journal of Mathematics},
    VOLUME = {304},
      YEAR = {2020},
    NUMBER = {1},
     PAGES = {347--384},
      ISSN = {0030-8730,1945-5844},
   MRCLASS = {11R42 (11G15 11G35 11M20 11R29)},
  MRNUMBER = {4053203},
MRREVIEWER = {St\'ephane\ R.\ Louboutin},
       DOI = {10.2140/pjm.2020.304.347},
       URL = {https://doi.org/10.2140/pjm.2020.304.347},
}

@article{Yan25,
	title={Spectral Reciprocity: A {F}ourier--Analytic Approach},
	author={Yang, Liyang},
	journal={arXiv preprint arXiv:2512.03305},
	year={2025}
}

@article{Yan26,
  title={Relative Trace Formula and Twisted {$ L $}-functions: the {Burgess} Bound},
  author={Yang, Liyang},
  journal={arXiv preprint arXiv:2305.10719},
  note={to appear in Compositio Mathematic},
  year   = {2026}
}

@article{Yan26b,
  title={Symmetric Spectral Reciprocity for {$\mathrm{GL}_2$} and Uniform Subconvexity},
  author={Yang, Liyang},
  journal={arXiv preprint arXiv:2607.04476},
  year={2026}
}

@book {YZZ13,
    AUTHOR = {Yuan, Xinyi and Zhang, Shou-Wu and Zhang, Wei},
     TITLE = {The {G}ross-{Z}agier formula on {S}himura curves},
    SERIES = {Annals of Mathematics Studies},
    VOLUME = {184},
 PUBLISHER = {Princeton University Press, Princeton, NJ},
      YEAR = {2013},
     PAGES = {x+256},
      ISBN = {978-0-691-15592-0},
   MRCLASS = {11G18 (11F70 14G35)},
  MRNUMBER = {3237437},
MRREVIEWER = {Ernest\ Hunter\ Brooks},
}

\end{document}